\documentclass [12pt,twoside,a4paper]{article}
\usepackage{amsfonts}
\usepackage{amsthm}
\usepackage{amsmath}
\usepackage{amstext}
\usepackage{amssymb}
\usepackage{mathrsfs}
\usepackage{amscd}
\usepackage{xypic}
\usepackage{epsf}              
\usepackage{fancybox}          
\usepackage{color}             
\usepackage{fancyhdr}
\usepackage[hang,footnotesize]{caption2}  
\usepackage{enumerate} 
\usepackage{float}
\usepackage{makecell}
\usepackage{subfigure}
\usepackage{graphicx}
\usepackage{epstopdf}
\usepackage{epsfig}
\usepackage[numbers]{natbib}
\usepackage{ragged2e}
\usepackage{graphicx,epsfig,epstopdf,subfigure}
\numberwithin{equation}{section}
\newfont{\aaa}{cmb10 at 19pt}
\newfont{\bbb}{cmb10 at 14pt}
\newtheorem{Case}{Case}

\newtheorem{Subcase}{Subcase}[Case]
\newtheorem{theorem}{Theorem}[section]
\newtheorem{corollary}[theorem]{Corollary}
\newtheorem{lemma}[theorem]{Lemma}
\newtheorem{claim}[theorem]{Claim}
\newtheorem{definition}[theorem]{Definition}

\makeatletter

\newcommand{\Rmnum}[1]{\expandafter\@slowromancap\romannumeral #1@}
\makeatother

\newcommand{\beq}{\begin{equation}}
\newcommand{\eeq}{\end{equation}}
\newcommand{\bey}{\begin{eqnarray}}
\newcommand{\eey}{\end{eqnarray}}
\newcommand{\beyy}{\begin{eqnarray*}}
\newcommand{\eeyy}{\end{eqnarray*}}

\bfseries\normalfont

\UseRawInputEncoding
\begin{document}

\title{Cycle-factors in oriented graphs\thanks{The author's work is supported by NNSF of China (No.12071260)}}

\author{Zhilan Wang, Jin Yan\thanks{Corresponding author. E-mail address: yanj@sdu.edu.cn}, Jie Zhang \unskip\\[2mm]
School of Mathematics, Shandong University, Jinan 250100, China}

\date{}
\maketitle
\begin{abstract}
Let $k$ be a positive integer. A $k$-cycle-factor of an oriented graph is a set of disjoint cycles of length $k$ that covers all vertices of the graph. In this paper, we prove that there exists a positive constant $c$ such that for $n$ sufficiently large, any oriented graph on $n$ vertices with both minimum out-degree and minimum in-degree at least $(1/2-c)n$ contains a $k$-cycle-factor for any $k\geq4$. Additionally, under the same hypotheses, we also show that for any sequence $n_1, \ldots, n_t$ with $\sum^t_{i=1}n_i=n$ and the number of the $n_i$ equal to $3$ is $\alpha n$, where $\alpha$ is any real number with $0<\alpha<1/3$, the oriented graph contains $t$ disjoint cycles of lengths $n_1, \ldots, n_t$. This conclusion is the best possible in some sense and refines a result of Keevash and Sudakov. 
\end{abstract}

\section{Introduction}
Researchers have made a fruitful study on tilings in graphs. Among them, the most outstanding
result, Dirac's Theorem \cite{Dirac}, asserts that any graph $G$ on $n\geq3$ vertices with minimum degree at least $n/2$ contains a \emph{Hamiltonian cycle}, which is a cycle passing through every vertex of $G$. This theorem raises the crucial question of determining what minimum degree is needed to identify a specific structure in a graph.
The Hajnal-Szemer\'{e}di theorem \cite{Hajnal} is a significant result in this regard, which shows that for every positive integer $t$,
if $G$ is a graph on $n$ vertices with $n\equiv0\ (\emph{\emph{mod}}\ t )$ such that the minimum degree of $G$ is at least $(1-1/t)n$, then $G$ contains a $K_t$-factor.

We naturally ask if these results have some similarities to \emph{oriented graphs}, which are orientations of simple graphs.
As any transitive tournament (with large minimum degree) contains no a cycle, we consider the \emph{minimum semi-degree} $\delta^0(D)=\min\{\delta^{+}(D), \delta^{-}(D)\}$ of an oriented graph $D$ instead of the minimum degree condition, where $\delta^{+}(D)$ (resp., $\delta^{-}(D)$) is the smallest value of the out-degrees (resp., in-degrees) of the vertices in $D$.

Thomassen \cite{Thomassen} posed a question about the analogue of Dirac's theorem in oriented graphs. He inquired about what minimum semi-degree forces a Hamiltonian cycle in an oriented graph. Since the question was raised, a series of improved bounds have been achieved in \cite{H1, H2, Kelly, T1, T2}. Subsequently, Keevash, K\"{u}hn and Osthus \cite{K1} provided a precise minimum semi-degree condition to prove the existence of Hamiltonian cycles, which answers Thomassen's question for oriented graphs with large order.
\begin{theorem}\emph{(Keevash et al. \cite{K1})}
There exists a number $n_0$ such that any oriented graph $D$ on $n\geq n_0$ vertices with $\delta^0(D)\geq\lceil(3n-4)/8\rceil$ contains a Hamiltonian cycle.
\end{theorem}

We further study cycle structure in oriented graphs. A number of (directed) graphs are said to be \emph{vertex-disjoint} (briefly \emph{disjoint}) if no two of them have any common vertex. Let $T$ and $D$ be directed graphs. A \emph{$T$-factor} of $D$ refers to a collection of disjoint copies of $T$ in $D$, covering all vertices of $D$. In particular, we call it a \emph{cycle-factor} if $T$ is a cycle. Additionally, if $T$ is a $k$-cycle, a cycle of length $k(\geq3)$, then we say it a \emph{$k$-cycle-factor}. Notably, a Hamiltonian cycle can be considered as a cycle-factor with exactly one cycle. Furthermore, the following theorem establishes the existence of cycle-factors with the same length in oriented graphs.
\begin{theorem}\label{t1}
Let $k$ be a positive integer with $k\geq4$. There exist positive constants $c$ and $n_0$ such that for every $n\geq n_0$ and $n\equiv0\ (mod\ k)$, the following holds. If $D$ is an oriented graph on $n$ vertices with $\delta^0(D)\geq(1/2-c)n$, then $D$ contains a $k$-cycle-factor.
\end{theorem}
Specially, Li and Molla \cite{Li} proved that an oriented graph $D$ has a $3$-cycle-factor if and only if $D$ does not have a divisibility barrier under the same hypotheses as in Theorem \ref{t1}. Here a \emph{divisibility barrier} of $D$ is a partition $\mathcal{P}=\{V_1, V_2, V_3\}$ of $V(D)$ with the following properties: (i) there is no arc from $V_j$ to $V_{j-1}$ for any $j\in[3]$ and $V_0=V_3$; and (ii) $|V_1|$, $|V_2|$ and $|V_3|$ are not all equivalent modulo $3$.

A tournament is \emph{regular} if every vertex has the same in-degree and out-degree. Together with the result of Li and Molla and Theorem \ref{t1}, we can establish the following corollary.
\begin{corollary}\label{qaz}
There exists a positive integer $n_0$ such that for any $k\geq3$ and every odd integer $n\geq n_0$ with $n\equiv0\ (mod\ k)$, any regular tournament on $n$ vertices contains a $k$-cycle-factor.
\end{corollary}

Further, we also present the following significant finding.
\begin{theorem}\label{songbaobao}
There exists a positive constant $c$ such that for $n$ sufficiently large, if $D$ is an oriented graph on $n$ vertices with $\delta^0(D)\geq(1/2-c)n$, then for any sequence $n_1, \ldots, n_t$ satisfying $\sum^t_{i=1}n_i=n$ and the number of the $n_i$ equal to $3$ is less than $\alpha n$, where $\alpha$ is any real number with $0<\alpha<1/3$, $D$ contains $t$ disjoint cycles of lengths $n_1, \ldots, n_t$.
\end{theorem}
Notice that the value of $\alpha$ is best possible in Theorem \ref{songbaobao} in some sense since, if $D$ contains a divisibility barrier, then $D$ does not have a $3$-cycle-factor whatever the minimum semi-degree of $D$ is. The above theorem refines a result of Keevash and Sudakov. 
\begin{theorem}\emph{(Keevash and Sudakov \cite{Keevash})} \label{KKK}
There exist constants $c^\prime, C>0$ such that for $n$ sufficiently large, if $D$ is an oriented graph on $n$ vertices with $\delta^0(D)\geq(1/2-c^\prime)n$ and $n_1, \ldots, n_t$ are numbers with $\sum_{i=1}^tn_i\leq n-C$, then $D$ contains disjoint cycles of length $n_1, \ldots, n_t$.
\end{theorem}
Combining with Corollary \ref{qaz} and Theorem \ref{songbaobao}, we can obtain the result as follows.
\begin{corollary}
There is $n_0$ such that for any positive integer partition $n=n_1+\cdots+n_t$, where $n_i\geq3$ for each $i\in[t]$, every regular tournament on $n\geq n_0$ vertices contains $t$ disjoint cycles of lengths $n_1, \ldots, n_t$.
\end{corollary}

The organization of the rest of the paper is as follows. In Section $2$, we prove Theorem \ref{songbaobao} by utilizing Theorem \ref{t1}. In Section $3$, we introduce some terminology and notation in Subsection $3.1$. Then we provide a brief outline of the proof of Theorem \ref{t1} in Subsection $3.2$, and present some results used in the proof of Theorem \ref{t1} in Subsection $3.3$. Finally, in Section $4$, we give the detailed proof of Theorem \ref{t1}.
\section{Proof of Theorem \ref{songbaobao}}
In this section, we apply Theorem \ref{t1} to prove Theorem \ref{songbaobao}. Further, our proof requires the lemma and theorems below. Lemma \ref{songbao1} is derived from the ``Chernoff bound'' approximation: $\mathbb{P}(|X-\mathbb{E}X|>a\mathbb{E}X)<2e^{-\frac{a^2}{3}\mathbb{E}X}$, where $0<a<3/2$ and $X$ is a hypergeometric random variable with expectation $\mathbb{E}X$. Additionally, let $D$ be an oriented graph and $U$ be a vertex set of $D$, we use $D[U]$ to represent the subdigraph of $D$ induced by $U$. For a vertex $u$ in $D$, we also denote by $d^+(u, U)$ and $d^-(u, U)$ the out-degree and the in-degree of $u$ in $U$, respectively.
\begin{lemma}\label{songbao1}\emph{(Keevash and Sudakov \cite{Keevash} - Lemma $3.2$)}
For any $\alpha, \beta>0$ there exists a number $n_0$ such that the following holds. Suppose $D$ is an oriented graph on $n\geq n_0$ vertices with $\delta^0(D)\geq\alpha n$ and $\beta n\leq s\leq(1-\beta)n$. Then there exists a partition of $V(D)$ as $A\cup B$ with $|A|=s$ and $|B|=n-s$ such that $\delta^0(D[ A])\geq(\alpha-n^{-1/3})s$ and $\delta^0(D[B])\geq(\alpha-n^{-1/3})(n-s)$.
\end{lemma}

Kelly, K\"{u}hn and Osthus showed a precise lower bound on the minimum
semi-degree required to guarantee the existence of an $l$-cycle for all $3\leq l\leq|V(D)| $ in a sufficiently large oriented graph $D$.
\begin{theorem} \emph{(Kelly et al. \cite{Kelly})} \label{songbao3}
There is an integer $n_0$ such that every oriented graph $D$ on $n\geq n_0$ vertices with $\delta^0(D)\geq(3n-4)/8$ contains a $k$-cycle for each $k\in\{3, \ldots, n\}$.
\end{theorem}
Keevash and Sudakov stated a result of the almost $3$-cycle-factor under the same minimum semi-degree condition as in Theorem \ref{songbaobao}.
\begin{theorem} \emph{(Keevash and Sudakov \cite{Keevash})} \label{KK}
There is some real $c_1>0$ so that for sufficiently large $n$, any oriented graph $D$ on $n$ vertices with $\delta^0(D)\geq(1/2-c_1)n$ contains disjoint $3$-cycles covering all but at most $3$ vertices.
\end{theorem}
They also gave a precise minimum semi-degree
bound to force some cycle-factors in an oriented graph.
\begin{theorem} \emph{(Keevash and Sudakov \cite{Keevash})} \label{songbao2}
For any $\delta>0$ there exist numbers $M$ and $n_0$ such that if $D$ is an oriented graph on $n>n_0$ vertices with $\delta^0(D)\geq(3/8+\delta)n$ and $n_1, \ldots, n_t$ are numbers satisfying $n_i\geq M$ for any $i\in[t]$ and $\sum^t_{i=1}n_i=n$, then there is a cycle-factor along with cycle lengths $n_1, \ldots, n_t$ in $D$.
\end{theorem}
\noindent \textbf{Proof of Theorem \ref{songbaobao}.} Choose $M$ such that Theorem \ref{songbao2} applies with $\delta=1/10$. Then, choose $c^\prime$ with $1/81>c^\prime>4(M-2)n^{-1/3}$ such that Theorem \ref{t1} holds for all $k\geq4$ and sufficiently large $n$ with the parameter $c^\prime$ playing the role of $c$.
Select the real number $c_1$ with $c_1\geq c^\prime/2$ ensuring the conditions for Theorem \ref{KK} are satisfied.

Suppose that $n=n_1+\cdots+n_t$ is large enough, in which the number of the $n_i$ equal to 3 is $\alpha n$, where $\alpha$ is any real number with $0<\alpha<1/3$. Consider $D$ to be an oriented graph on $n$ vertices with $\delta^0(D)\geq(1/2-c)n$, where $c<c^\prime/2-n^{-1/3}$. 
For $3\leq k\leq M$, let $N_k$ be the number of the $n_i$ equal to $k$, and assume $N_L=\sum_{n_i>M}n_i$. By Lemma \ref{songbao1}, there exists a partition of $V(D)$ as $X\cup Y$, where $|X|=3N_3+3$ and $Y=n-(3N_3+3)$, such that $\delta^0(D[X])\geq(1/2-c-n^{-1/3})(3N_3+3)$ and $\delta^0(D[Y])\geq(1/2-c-n^{-1/3})(n-(3N_3+3))$. Then Theorem \ref{KK} suggests that $D[X]$ contains $N_3$ disjoint $3$-cycles (denoted by $\mathcal{C}_3$) with replacing $c_1$ in Theorem \ref{KK} with $c+n^{-1/3}$ since $c+n^{-1/3}<c^\prime/2\leq c_1$. Let $\{x_1, x_2, x_3\}=X\setminus V(\mathcal{C}_3)$ and $Y^\prime=Y\cup\{x_1, x_2, x_3\}$. Let $c_2$ be a real number with $c_2>\frac{c+3\alpha(c+n^{-1/3})}{1-3\alpha}$. Clearly $c_2<1/80$. Then we have $$d^{\sigma}(x_i,  Y^\prime)\geq(1/2-c)n-(1/2+c+n^{-1/3})(3N_3+3)\geq(1/2-c_2)|Y^\prime|$$
for any $i\in[3]$ and $\sigma\in\{+, -\}$. Also, for any $y\in Y$,
$$ d^{\sigma}(y, Y^\prime)\geq(1/2-c-n^{-1/3})(n-(3N_3+3))\geq(1/2-c_2)|Y^\prime|.$$
In the following, we will get all the other cycles of length more than $3$ in $D[Y^\prime]$. For the sake of convenience, let $D^\prime=D[Y^\prime]$ and $n^\prime=|Y^\prime|$.

If there exists a $k$ such that $N_k<\frac{c_2n^\prime}{3M^2}$ or $N_L<\frac{c_2n^\prime}{4}$, then we can repeatedly use Theorem \ref{songbao3} to select $N_k$ disjoint $k$-cycles or all cycles of lengths more than $M$ as desired, respectively. This is because by $c_2<1/80$, we have that
\begin{align}\label{ssss}
(1/2-c_2)n^\prime-\sum^M_{k=4}k\cdot\frac{c_2n^\prime}{3M^2}-\frac{c_2n^\prime}{4}&\geq(1/2-c_2)n^\prime-\frac{c_2n^\prime}{4}-\frac{c_2n^\prime}{4}
>(3n^\prime-4)/8.
\end{align}
Also, after picking all cycles for any $k$ with $N_k<\frac{c_2n^\prime}{3M^2}$ or $N_L<\frac{c_2n^\prime}{4}$, we will be left with minimum semi-degree at least $(\frac{1}{2}-\frac{3c_2}{2})n^\prime$ by (\ref{ssss}). By applying Lemma \ref{songbao1} repeatedly, We may get a partition of the remaining vertices into $M-2$ disjoint vertex sets $V_4, \ldots, V_M, V_L$ such that for any $i\in\{4, \ldots, M, L\}$
\begin{align*}
\begin{split}
|V_i|=\left \{
\begin{array}{ll}
iN_i, &\emph{\emph{if}}\ N_i\geq \frac{c_2n^\prime}{3M^2},\ \emph{\emph{for}}\ i\in\{4, \ldots, M\}, \\
N_L, & \emph{\emph{if}}\ N_L\geq \frac{c_2n^\prime}{4},\ \emph{\emph{for}}\ i=L,\\
0, &\emph{\emph{otherwise}},
\end{array}
\right.
\end{split}
\end{align*}
and $\delta^0(D[ V_i])\geq(1/2-3c_2/2-(M-2)(n^\prime)^{-1/3})|V_i|\geq(1/2-2c_2)|V_i|\geq(3/8+\delta)|V_i|$.
Then, according to Theorem \ref{t1}, we can embed $N_i$ disjoint $i$-cycles in $D[V_i]$ for any $i\in\{4, \ldots, M\}$. Also, we may embed all ``long'' cycles in $D[V_L]$ by using Theorem \ref{songbao2}.\hfill $\Box$

\smallskip

\emph{Remark.} If the number of $3$-cycles is a constant, that is, it is independent of the cardinality of the oriented graph $D$, then we can first greedily find $N_3$ disjoint $3$-cycles by Theorem \ref{songbao3} and the bound of $\delta^0(D)$. Then in the remaining oriented subgraph, we use the proof of Theorem \ref{songbaobao} to find all the other cycles.
\section{Preparation for Theorem \ref{t1}}

\subsection{Terminology and notation}
Let $D=(V, A)$ be an oriented graph and let $r$ and $s$ be two distinct vertices of $D$. The notation $rs$ is defined as an arc from $r$ to $s$. The \emph{out-neighbour} (resp., \emph{in-neighbour}) of $s$ is defined by $N^{+}(s)=\{r\in V: sr\in A\}$ (resp., $N^{-}(s)=\{w\in V: ws\in A\}$). The \emph{out-degree} (resp., \emph{in-degree}) of $s$, which is denoted by $d^{+}(s)$ (resp., $d^{-}(s)$), is the cardinality of $N^{+}(s)$ (resp. $N^{-}(s)$)
that is, $d^{+}(s)=|N^{+}(s)|$ (resp. $d^{-}(s)=|N^{-}(s)|$).
Further, the \emph{degree} of $s$ in $D$ is given by $d_D(s)=d^+(s)+d^-(s)$. We define the \emph{minimum total degree} of $D$ to be $\delta(D)=\min_{s\in V}d_D(s)$.

For any subset $R\subseteq V$, define $N^\sigma(r, R)=N^\sigma(r)\cap R$ and $d^\sigma(r, R)=|N^\sigma(r, R)|$ for any $\sigma\in\{-, +\}$. We use the notation $e(R)$ to present the number of arcs in $D[R]$. We refer to $R$ as an \emph{$i$-set} if its cardinality $|R|$ is equal to $i$. Let $D-R=D[V\setminus R]$ and  $\overline{R}=V\setminus R$. For another vertex set $S$ in $V$ that is not necessarily disjoint to $R$, we denote by $e^{+}(R, S)$ (resp., $e^{-}(R, S)$) the number of arcs from $R$ to $S$ (resp., from $S$ to $R$). If $R=\{s\}$, we simply write $s$ instead of $\{s\}$ throughout the paper.

The \emph{$\gamma$-out-neighborhood} (resp., \emph{$\gamma$-in-neighborhood}) of $R$, denoted by $N_\gamma^{+}(R)$ (resp., $N_\gamma^{-}(R)$), to be the set of vertices $r$ in $\overline{R}$ such that $d^+(R, r)\geq |R|-\gamma n$ (resp., $d^-(R, r)\geq |R|-\gamma n\}$), that is, $N_\gamma^{+}(R)=\{r: r\in\overline{R}\ \emph{\emph{and}} \ d^+(R, r)\geq |R|-\gamma n\}$ (resp., $N_\gamma^{-}(R)=\{r: r\in\overline{R}\ \emph{\emph{and}} \ d^-(R, r)\geq |R|-\gamma n\}$). In addition, for a vertex set $S$ in $\overline{R}$, we define $N_{\gamma, S}^{+}(R)=N_\gamma^{+}(R)\cap S$ (resp., $N_{\gamma, S}^{-}(R)=N_\gamma^{-}(R)\cap S$).

\smallskip

The following definition gives a special partition of the vertex set $S\subseteq V$, which is a crucial structure obtained in the proof of Theorem \ref{t1}. Let $d$ be a positive integer, $\eta, \delta$ and $\beta$ be positive real numbers, and $n=|V|$.

\begin{definition}\emph{(}The $\delta$-extremal partition\emph{)}\label{definition3}
\emph{A partition $\mathcal{P}=\{V_1, \ldots, V_d\}$ of $S$ is \emph{$\delta$-extremal} if $|V_i|\geq\eta n$ for every $i\in[d]$, and there exists $A\in \mathcal{P}$ such that the new partition $\{A\cup R, S^-, S^+\}$ of $S$ satisfies the following statements, where $S^-=N_{\gamma, S}^-(A)$, $S^+=N_{\gamma, S}^+(A)$ and $R=S\setminus(A\cup S^-\cup S^+)$.\\
(i) $(1/3-100k^3\delta)n\leq|A\cup R|, |S^-|, |S^+|\leq (1/3+50k^3\delta)n$, and\\
(ii) the number of arcs from $A\cup R$ to $S^-$, from $S^-$ to $S^+$ and from $S^+$ to $A\cup R$ are each at most $\delta n^2$.}
\end{definition}

We define the number of vertices of a path as its \emph{length} and a \emph{$k$-path} is a path of length $k$. We often represent a $k$-path (-cycle) $C$ as $v_1v_2\cdots v_k(v_1)$ when $V(C)=\{v_1, v_2, \ldots, v_k\}$. For $V_1, \ldots, V_k\subseteq V$, $k$-cyc($V_1, \ldots, V_k$) counts the number of $k$-cycles with the vertex set $\{v_1, \ldots, v_k\}$ such that $v_i\in V_i$ for each $i$. We can abbreviate $k$-cyc($V_1, \ldots, V_k$) as $k$-cyc($V_1^t, V_{t+1}, \ldots, V_k$) when $V_1=V_2=\ldots=V_t$, where $t\in\{1, \ldots, k\}$. In particular, if $t=k$, then we write $k$-cyc($V_1$) instead of $k$-cyc($V_1^k$).  For a positive integer $s$, we simply write $[s]$ for $\{1, \ldots , s\}$.

The proof of Theorem \ref{t1} is done with the help of hypergraphs. For an integer $k\geq2$, a \emph{$k$-uniform hypergraph} is defined as a pair $H=(V(H), E(H))$ in which $V(H)$ is a finite set of vertices and  $E(H)=\{e\subseteq V(H): |e|=k\}$ is a set of $k$-element subsets of $V(H)$, whose members are called the \emph{edges} of $H$. Further, for a vertex $r\in V(H)$, let $N(r)$ denote the neighbour set of $r$ in $H$, that is, $N(r)=\{S\in[V(H)]^{k-1}: S\cup\{r\}\in E(H)\}$. The \emph{degree} of $r$ in $H$ is denoted as $d(r)=|N(r)|$. The \emph{minimum degree} of $H$ is $\delta_1(H)=\min_{r\in V(H)}d(r)$. 

\medskip

In the following, we construct a $k$-uniform hypergraph of an oriented graph. 
\begin{definition}\emph{(}The $k$-uniform hypergraph of $D$\emph{)}
\emph{For every oriented graph $D$, the \emph{$k$-uniform hypergraph $H(D)$} is a $k$-uniform hypergraph with the vertex set $V(D)$ and the $k$-set $\{v_1, \ldots, v_k\}$ is an edge in $H(D)$ if and only if it spans a $k$-cycle in $D$.}
\end{definition}
It is evident that there is a perfect matching in $H(D)$ if and only if the oriented graph $D$ has a $k$-cycle-factor.

\smallskip

We also need to introduce the following definitions that mainly come from the works of Han \cite{HH}, which build on the absorbing method initiated by R\"{o}dl, Ruci\'{n}ski, and Szemer\'{e}di \cite{Szemeredi}. Let $H=(V(H), E(H))$ be a $k$-uniform hypergraph. For a positive integer $r$, we call an additive subgroup of $\mathbb{Z}^r$ a \emph{lattice}. Let $\mathcal{P}=\{V_1, \ldots, V_r\}$ be a partition of $V(H)$. The \emph{index vector} $\mathbf{i}_{\mathcal{P}}(R)\in \mathbb{Z}^r$ of a subset $R\subseteq V(H)$ with respect to $\mathcal{P}$ is denoted by $\mathbf{i}_{\mathcal{P}}(R)=(|R\cap V_1|, \ldots, |R\cap V_r|)$. We define $I_{\mathcal{P}}(H)=\{\mathbf{i}_{\mathcal{P}}(e): e\in E(H)\}$ and
$L_{\mathcal{P}}(H)$ as the lattice generated by $I_{\mathcal{P}}(H)$.
For any $\mu>0$, we call $I_{\mathcal{P}}^{\mu}(H)$ to be the set of all $\mathbf{i}\in \mathbb{Z}^r$ such that there are at least $\mu n^k$ edges $e\in E(H)$ satisfying $\mathbf{i}_{\mathcal{P}}(e)=\mathbf{i}$. Similarly, we denote by $L_{\mathcal{P}}^{\mu}(H)$ the lattice in $\mathbb{Z}^r$ generated by $I_{\mathcal{P}}^{\mu}(H)$.
The \emph{ith unit vector}, denoted by $\mathbf{u}_i$, is the vector in which the coordinate $i$ is $1$ and all other coordinates are $0$.
We call a non-zero difference $\mathbf{u}_i-\mathbf{u}_j$ for distinct $i, j\in[r]$ a \emph{transferral} in $\mathbb{Z}^r$. Further,
a \emph{$2$-transferral} $\mathbf{v}\in L_{\mathcal{P}}^{\mu}(H)$ is a transferral for which there exist $\mathbf{v}_1, \mathbf{v}_2\in I_{\mathcal{P}}^{\mu}(H)$
such that $\mathbf{v}=\mathbf{v}_1-\mathbf{v}_2$.
In addition, $L_{\mathcal{P}}^{\mu}(H)$ is called \emph{$2$-transferral-free} if it contains no a $2$-transferral.

Let $\alpha, \beta>0$ be real numbers and $k, l, r>0$ be integers. We define two vertices $u$ and $v$ in $V(H)$ to be \emph{\emph{(}$H, \beta, l$\emph{)}-reachable} if the number of ($kl-1$)-sets $R$ such that both $H[R\cup\{u\}]$ and $H[R\cup\{v\}]$ have perfect matchings, is at least $\beta n^{kl-1}$. We also define $\widetilde{N}_{H, \alpha, 1}(v)$ as the set of vertices that are $(H, \alpha, 1)$-reachable to $v$. A vertex set $U\subseteq V(H)$ is considered \emph{\emph{(}$H, \beta, l$\emph{)}-closed} if for any two distinct vertices $u$ and $v$ in $U$, they are ($H, \beta, l$)-reachable. Moreover, a partition $\mathcal{P}=\{V_1, \ldots, V_r\}$ of $V(H)$ is called a \emph{\emph{(}$H, \beta, l$\emph{)}-closed partition} if $V_i$ is ($H, \beta, l$)-closed for every $i\in[r]$.

In this paper, we use the notation $\alpha\ll\beta$ to indicate that $\alpha$ is chosen to be sufficiently small relative to $\beta$ such that all calculations needed in our proof are valid. For the real numbers $a$ and $b$, we use $a\pm b$ to represent an unspecified real number in the interval $[a-b, a+b]$.
\subsection{Proof overview of Theorem \ref{t1}}
In this subsection, we provide a brief outline of the proof of Theorem \ref{t1} in order to give readers a general understanding of the proof of Theorem \ref{t1}.

Let $D$ be an oriented graph as stated in Theorem \ref{t1} and $H(D)$ be the $k$-uniform hypergraph of $D$. By utilizing Lemmas \ref{l3}-\ref{l12} in Subsection $3.3$, we will obtain a vertex set $S\subseteq V(D)$ such that it covers all vertices in $V(D)$ expect for the vertices in $V_0$ and satisfies Lemma \ref{l13}. Additionally, combining with Lemma \ref{l3} again, we could get a partition of $S$ into $V_1, \ldots, V_r$ such that for any $i\in[r]$, $|V_i|\geq\eta n$ and $V_i$ is $(H(D), \beta, l)$-closed in $H(D)$, where the integers $r, l\geq1$. Let $\mathcal{P}^\prime=\{V_0, V_1, \ldots, V_r\}$. We divide the proof in Theorem \ref{t1} into two cases based on whether $S$ is $(H(D), \beta, l)$-closed.

The proof in the case when $S$ is $(H(D), \beta, l)$-closed mainly relies on the lattice-based absorbing method developed recently by Han \cite{Ding}. The absorbing technique initiated by R\"{o}dl, Ruci\'{n}ski, and Szemer\'{e}di \cite{Szemeredi} has been shown to be efficient on finding spanning structures in graphs and hypergraphs. The lemma below serves as our absorbing lemma.
\begin{lemma}\emph{(Ding et al. \cite{Ding} - Lemma $2.9$)} \label{l5}
Let $\beta$ and $\eta$ be real numbers with $0<\beta, \eta<1$ and let $l$ and $k$ be positive integers. There exists $\omega>0$ such that the following holds for the sufficiently large integer $n$. Assume that $H$ is a $k$-uniform hypergraph on $n$ vertices with characteristics as follows.\\
(1) For any $v\in V(H)$, there are at least $\eta n^{k-1}$ distinct edges containing it.\\
(2) There exists $V_0\subseteq V(H)$ with $|V_0|\leq \eta^2 n$ such that $V(H)\setminus V_0$ is $(H, \beta, l)$-closed.\\
Then there exists a vertex set $U$ with $V_0\subseteq U\subseteq V(H)$ and $|U|\leq\eta n$ such that for any vertex set $W\subseteq V(H)\setminus U$ with $|W|\leq\omega n$ and $|W|\equiv0\ (\mbox{mod}\ k)$, both $H[U]$ and $H[U\cup W]$ contain perfect matchings.
\end{lemma}
Otherwise, $S$ is not ($H(D), \beta, l$)-closed. In this case, if $L_{\mathcal{P}^\prime}^\mu(H(D))$ contains a $2$-transferral $\mathbf{u}_i-\mathbf{u}_j$ for distinct $i, j\in[r]$, then merge the partite sets $V_i$ and $V_j$ and think about the new partition $\mathcal{P}^\prime-V_i-V_j+(V_i\cup V_j)$. By applying Lemma \ref{l6} that is stated below in Subsection $3.3$, continuously merge the partite sets corresponding to $2$-transferrals until we obtain a final partition $\mathcal{P}=\{V_0, V_1, \ldots, V_d\}$ of $V(D)$ such that $\mathcal{P}$ contains no $2$-transferral $\mathbf{u}_i-\mathbf{u}_j$ for distinct $i, j\in[d]$ and $V_i$ is closed for any $i\in [d]$. We can assume that $d\geq2$.

In the subsequent section, we will demonstrate that there exists a partite set $A\in S$ with $k$-cyc$(A^{k-1}, S\setminus A)\leq\alpha n^k$ in Claim \ref{a1} in Subsection 4.2. Additionally, it will follow from Lemma \ref{l14} that $\mathcal{P}\setminus V_0$ forms a $\delta$-extremal partition of $S$ in Claim \ref{a2}. Let $a_1=(k-1)/3$, $a_2=(k-2)/3$ and $a_3=k/3$ for convenience. Then by good properties of this $\delta$-extremal partition, we will prove that there exist disjoint partite subsets $A_1, A_2, A_0\in\mathcal{P}\setminus V_0$ that satisfy the following conclusions.
\begin{align*}\label{aa2}
\begin{split}
\left \{
\begin{array}{lr}
  k\emph{\emph{-cyc}}(A_1^{a_1+1}, A_2^{a_1}, A_0^{a_1}), k\emph{\emph{-cyc}}(A_1^{a_1}, A_2^{a_1+1}, A_0^{a_1})\geq\mu n^k,                    & \emph{\emph{when}}\ k\equiv1\ (\emph{\emph{mod}}\ 3),\\
   k\emph{\emph{-cyc}}(A_1^{a_2+2}, A_2^{a_2}, A_0^{a_2}),
   k\emph{\emph{-cyc}}(A_1^{a_2+1}, A_2^{a_2+1}, A_0^{a_2})\geq\mu n^k,     & \emph{\emph{when}}\ k\equiv2\ (\emph{\emph{mod}}\ 3),\\
 k\emph{\emph{-cyc}}(A_1^{a_3}, A_2^{a_3}, A_0^{a_3})\geq\mu n^k,  k\emph{\emph{-cyc}}(A_1^{a_3+1}, A_2^{a_3}, A_0^{a_3-1})\geq\mu n^k,     & \emph{\emph{when}}\ k\equiv0\ (\emph{\emph{mod}}\ 3).\\
\end{array}
\right.
\end{split}
\end{align*}
Based on these conclusions, we can conclude that when $k\equiv1$, or $2$ (mod $3$), we can merge $A_1$ and $A_2$. Similarly, when $k\equiv0$ (mod $3$), we can merge $A_1$ and $A_0$. This leads to a contradiction and hence Theorem \ref{t1} holds when $S$ is not ($H(D), \beta, l$)-closed.
\subsection{Tools when $S$ is not ($H(D), \beta, l$)-closed}

\begin{lemma}\label{a3}
Assume that $c$ is a real number less than $1/2$. Let $D$ be an oriented graph on $l$ vertices with $\delta(D)\geq(1-2c)l$. Then for any positive integer $s$, $D$ has at least $s$ vertices of out-degree at least $(l-s-2cl)/2$ in $D$ and at least $s$ vertices of in-degree at least $(l-s-2cl)/2$ in $D$. Moreover, there exists a vertex $v$ in $D$ satisfying $l/4-cl\leq d^+(v)\leq 3l/4+cl$.
\end{lemma}
\begin{proof}
Let $v_1, \ldots, v_{l}$ be a vertex ordering of $D$ such that $d^+(v_1)\geq\cdots\geq d^+ (v_{l})$. Set $D^\prime=D[\{v_s, \ldots, v_{l}\}]$. By the degree condition of $D$, we obtain that in $D^\prime$, every vertex is not joined to at most $2cl$ vertices. Since $d^+(v_s)\geq\cdots\geq d^+ (v_{l})$, we get that $$d^+(v_s)\cdot|D^\prime|\geq\sum_{i=s}^ld^+(v_i)\geq e(D^\prime)=\sum_{i=s}^ld^+_{D^\prime}(v_i)\geq\frac{|D^\prime|\cdot(|D^\prime|-2cl)}{2},$$ which implies $d^+(v_s)\geq(|D^\prime|-2cl)/2\geq(l-s-2cl)/2$. So $D$ contains $s$ vertices of out-degree at least $(l-s-2cl)/2$ in $D$. Similarly, $D$ contains $s$ vertices of in-degree at least $(l-s-2cl)/2$ in $D$. This means that if let $s=\lfloor l/2\rfloor$, then we get that there are at least $\lfloor l/2\rfloor$ vertices with out-degree at least $(l-\lfloor l/2\rfloor-2cl)/2\geq l/4-cl$, and if let $s=\lceil l/2\rceil+1$, then there are $\lceil l/2\rceil+1$ vertices with in-degree at least $(l-\lceil l/2\rceil-1-2cl)/2\geq l/4-1-cl$. This yields that $D$ contains a vertex $v$ with out-degree at least $l/4-cl$ and in-degree at least $ l/4-1-cl$, which implies $l/4-cl\leq d^+(v)\leq(l-1)-(l/4-1-cl)=3l/4+cl$.
\end{proof}
\begin{lemma}\label{lx2} Let $c$ be a positive constant with $c<\frac{1}{2200}$, and $D$ be an oriented graph on $n$ vertices with $\delta^0(D)\geq(1/2-c)n$. For any $U\subseteq V(D)$, the following statements hold.\\
$(i)$ $e^{+}(U, \overline{U}),\ e^{-}(U, \overline{U})=\frac{|U|\cdot|\overline{U}|}{2}\pm c|U|n$;\\
$(ii)$ $e(U)\geq\frac{|U|(|U|-2cn)}{2}$;\\
$(iii)$ Let $S\subseteq V(D)$, where $S$ and $U$ are not necessarily disjoint. If $|S|, |U|\geq (1/2-c)n$, then there are at least $n^2/ 41$ arcs from $S$ to $U$.
\end{lemma}
\begin{proof}
(i) The proof for (i) has been previously provided in \cite{Li}. However, for the completeness of the paper, we present a straightforward proof of it here.

On the one hand, by the lower bound of $\delta^0(D)$, for any $\sigma\in\{+, -\}$, we have that
\begin{flushleft}
$e^\sigma(U, \overline{U})=\sum_{u\in U}d^\sigma(u)-d^\sigma(u, U)\geq$
\end{flushleft}
\begin{flushright}
$\sum_{u\in U}d^\sigma(u)-\tbinom{|U|}{2}\geq|U|\cdot(1/2-c)n-\frac{|U|^2}{2}=\frac{|U|\cdot|\overline{U}|}{2}-c|U|n.$
\end{flushright}
On the other hand, together with the fact that $e^+(U, \overline{U})+e^-(U, \overline{U})\leq|U|\cdot|\overline{U}|$, we get that $e^\sigma(U, \overline{U})\leq\frac{|U|\cdot|\overline{U}|}{2}+c|U|n$ for any $\sigma\in\{+, -\}$. Hence (i) holds.

\smallskip

(ii) is easily derived from the minimum semi-degree condition of $D$.

\smallskip

(iii) If necessary, we may assume that $|S|=|U|=(1/2-c)n$ by deleting the vertices. On the one hand, suppose that $|S\cap U|>n/ 5$, then by (ii) and $c<\frac{1}{2200}$, it is obtained that
\begin{equation*}
\begin{split}
e^{+}(S, U)\geq e(S\cap U)\geq\frac{|S\cap U|\cdot(|S\cap U|-2cn)}{2}\geq \frac{n}{5}\cdot\left(\frac{n}{5}-2cn\right)>\frac{n^2}{41},
\end{split}
\end{equation*}
Otherwise, $|S\cap U|\leq n/ 5$. Then by (i) of this lemma, $e^{+}(S, U)\geq e^{+}(S, \overline{S})-|S|\cdot|\overline{S\cup U}|$, $|\overline{S\cup U}|=n-|S\cup U|=n-(|S|+|U|-|S\cap U|)$ and $|S|=|U|=(1/2-c)n$, we have
\begin{equation*}
\begin{split}
e^{+}(S, U)\geq\frac{(1/2-c)^2n^2}{2}-(1/2-c)n\cdot(2cn+|S\cap U|)
\geq\left(\frac{1}{40}+\frac{5c^2}{2}-\frac{13c}{10}\right)n^2
\geq\frac{n^2}{41},
\end{split}
\end{equation*}
where the last inequality is due to $c<\frac{1}{2200}$. This complete the proof of the lemma.
\end{proof}
Adler, Alon and Ross \cite{Adler} proved that
any tournament of order $n$ has at most $\frac{n^{3/2}(n-1)!}{2^n}$ distinct Hamiltonian cycles. This implies that any oriented graph on $k$ vertices contains at most $\frac{k^{3/2}(k-1)!}{2^k}$ distinct $k$-cycles.
we will use this fact to prove the lemma in what follows. We say that two cycles $C_1$ and $C_2$ are \emph{strongly distinct} if $V(C_1)\neq V(C_2)$.
\begin{lemma}\label{l3}
Let $c$ be a constant with $c<1/2$ and $D$ be an oriented graph on $n$ vertices with $\delta^0(D)\geq(1/2-c)n$. Let $H(D)$ be the $k$-uniform hypergraph of $D$. Then the following statements hold.\\
$(i)$ Any $x\in V(D)$ belongs to at least $\frac{n^{k-1}}{6k^{1/2}k!}$ strongly distinct $k$-cycles. This implies that
$\delta_1(H(D))\geq \frac{n^{k-1}}{6k^{1/2}k!}\geq\frac{1}{6k^{3/2}}\cdot\binom{n-1}{k-1}.$\\
$(ii)$ Every set of $6k^{3/2}+1$ vertices in $H(D)$ contains two vertices that are $(H(D), \alpha, 1)$-reachable, where $\alpha$ is a real number with $0<\alpha<\frac{1}{108k^4\cdot k!}$.
\end{lemma}
\begin{proof}
We first give the proof of (i). To construct $k$-cycles containing the vertex $x$, we start by selecting as many $(k-2)$-paths as passible that begin at $x$, which gives at least $(1/2-c)n\cdot\prod_{i=0}^{k-5}((1/2-c)n-i)$ such ($k-2$)-paths. For any one of these $(k-2)$-paths $P=x\cdots y$, we may choose an arc from $N^{+}(y)$ to $N^{-}(x)$ without using any vertices in $V(P)$ to form a $k$-cycle. Obviously there are at most $(k-2)(n-1)$ arcs that repeat vertices in $V(P)$. Therefore, according to Lemma \ref{lx2}-(iii), there exist at least $\frac{n^2}{41}-(k-2)(n-1)$ arcs from $N^{+}(y)$ to $N^{-}(x)$ that are disjoint from $P$. So there are at least  $(\frac{n^2}{41}-(k-2)(n-1))\cdot(1/2-c)n\cdot\prod_{i=0}^{k-5}((1/2-c)n-i)\geq \frac{n^{k-1}}{48\cdot2^{k-3}}$ $k$-cycles through $x$. Recall that $x$ is in at most $\frac{k^{3/2}(k-1)!}{2^k}$ distinct $k$-cycles. Therefore, the number of strongly distinct $k$-cycles through $x$ is at least $\frac{2^k}{k^{3/2}(k-1)!}\cdot \frac{n^{k-1}}{48\cdot2^{k-3}}\geq \frac{n^{k-1}}{6k^{1/2}k!}.$

Secondly, we prove that (ii) holds.
Suppose not, let $X$ be a vertex set with $|X|=6k^{3/2}+1$ such that for any two distinct vertices $u, v\in X$, they are not $(H(D), \alpha, 1)$-reachable. This yields $|N(u)\cap N(v)|<\alpha n^{k-1}$ in $H(D)$.
Let $X=\{u_1, \ldots , u_{6k^{3/2}+1}\}$, then by (i) of this lemma, in $H(D)$ we have that
\begin{equation*}
\begin{split}
\binom{n-1}{k-1}&\geq|N(u_1)\cup N(u_2)\cup\cdots \cup N(u_{6k^{3/2}+1})|\\
&\geq \sum_{1\leq i\leq6k^{3/2}+1}|N(u_i)|-\sum_{1\leq i<j\leq6k^{3/2}+1}|N(u_i)\cap N(u_j)|\\
&\geq(6k^{3/2}+1)\cdot\frac{1}{6k^{3/2}}\cdot\binom{n-1}{k-1}-\alpha n^{k-1}\cdot\binom{6k^{3/2}+1}{2},\\
\end{split}
\end{equation*}
a contradiction since $\alpha<\frac{1}{108k^4\cdot k!}$. 
\end{proof}
The following lemma states that if every set of vertices of constant size in a hypergraph contains two reachable vertices, then we can find a very large set $S$ in which all vertices have large reachable neighborhoods.
\begin{lemma}\label{l12}\emph{(Ding et al. \cite{Ding} - Lemma $2.6$)}
Assume that $\delta, \alpha>0$ and integers $f, k\geq2$.
If $H$ is a $k$-uniform hypergraph on $n$ vertices such that every set of $f+1$ vertices contains two vertices that are $(H, \alpha, 1)$-reachable, then there exists $S\subseteq V(H)$ with cardinality $|S|\geq(1-f\delta)n$ such that for any $v\in S$, $|\widetilde{N}_{H, \alpha, 1}(v)\cap S|\geq\delta n$.
\end{lemma}
The following lemma is utilized to obtain a partition of a vertex set $S$ of the oriented graph that possesses desirable properties.
\begin{lemma}\emph{(Han and Treglown \cite{Han1} - Lemma $6.3$)} \label{l13}
Suppose that $\delta>0$ and integers $f, k\geq2$. Let $\alpha$ be a positive number with $\alpha\ll \delta$. Then there exists a constant $\beta>0$ such that the following statement holds for all sufficiently large $n$. Assume that $H$ is a $k$-uniform hypergraph on $n$ vertices and $S\subseteq V(H)$ satisfying $|\widetilde{N}_{H, \alpha, 1}(v)\cap S|\geq\delta n$ for any $v\in S$. Additionally, suppose that every set of $f+1$ vertices in $S$ contains two vertices that are $(H, \alpha, 1)$-reachable in $H$. Then we can partition $S$ into subsets $V_1, \ldots, V_r$ with $r\leq\min\{f, \lfloor1/\delta\rfloor\}$ such that for every $i\in[r]$, we have that $|V_i|\geq(\delta-\alpha)n$ and $V_i$ is $(H, \beta, 2^{f-1})$-closed in $H$.
\end{lemma}

The lemma below states that for a closed partition $\mathcal{P}=\{V_1, \ldots, V_r\}$ of the subset $S$ of an oriented graph $D$, the partition of $S$ formed by merging $V_i$ and $V_j$ is closed if $\mathbf{u}_i-\mathbf{u}_j$ is a 2-transferral in $L_{\mathcal{P}\cup V_0}^\mu(H(D))$ for $1\leq i<j\leq d$, where $V_0=V(D)\setminus S$. 
\begin{lemma}\emph{(Han \cite{HH} - Lemma $3.4$)} \label{l6}
Assume that $1/n\ll \beta\ll\beta^\prime, \mu\ll1/r, 1/l$. Let $D$ be an oriented graph and let $S\subseteq V(D)$ and $V_0=V(D)\setminus S$. Suppose that $\mathcal{P}=\{V_1, \ldots, V_r\}$ is a partition of $S$. If $\mathcal{P}$ is \emph{(}$H(D), \beta^\prime, l $\emph{)}-closed and there exists a $2$-transferral $\mathbf{u}_i-\mathbf{u}_j\in L_{\mathcal{P}\cup V_0}^\mu(H(D))$ for $i, j\in[r]$, then the partition of $S$ formed by merging $V_i$ and $V_j$ is \emph{(}$H(D), \beta, (k+1)l+1$\emph{)}-closed.
\end{lemma}
In the proof of Theorem \ref{t1}, in order to find as many distinct $k$-cycles as possible, we will frequently utilize the following lemma and some of the ideas presented in its proof. Before introducing the lemma, we provide the definition as follows.
\begin{figure*}[h]
\centering
\includegraphics[width=0.45\linewidth]{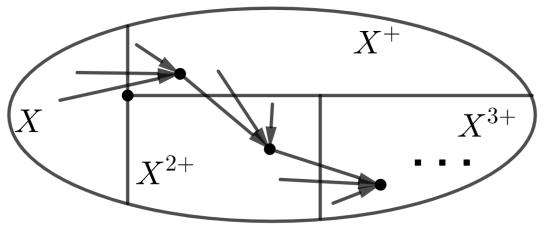}
\caption{Definition \ref{aaaaa} for $\sigma=+$.}
\label{b-1}
\end{figure*}
\begin{definition}\label{aaaaa}
\emph{Assume that $\epsilon$ is a real number with $\epsilon\ll1$. For any vertex subset $X$ of an oriented graph and every positive integer $i$, we denote that
\begin{center}
$X^{i\sigma}=\{x\in X^{(i-1)\sigma}:\ d^-(x, X^{(i-1)\sigma})\geq \epsilon n\}$ for any $\sigma\in\{+, -\},$
\end{center}
where $X^{0\sigma}=X$ and write $X^{1\sigma}$ as $X^\sigma$ for abbreviation. 
}
\end{definition}
\begin{lemma}\label{lx8}
Suppose that $1/n\ll c\ll\epsilon\ll\eta\ll1$. Let $D$ be an oriented graph on $n$ vertices with $\delta^0(D)\geq(1/2-c)n$. Then for any $X\subseteq V(D)$ with $|X|>\eta n$, we have that $|X^{i\sigma}|\geq|X^{(i-1)\sigma}|/ 2-\epsilon n$ and $|X^{i\sigma}|\geq|X|/ 2^{i}-2\epsilon n$ for any $\sigma\in\{+, -\}$.
\end{lemma}
\begin{proof}
We only give the proof of the case when $\sigma=+$, because the proof of the case when $\sigma=-$ is similar. We first prove $|X^{i+}|\geq|X^{(i-1)+}|/2-\epsilon n$. Let $X^\prime=X^{(i-1)+}\setminus X^{i+}$, then by Lemma \ref{lx2}-(ii), we get
\begin{equation*}
\begin{split}
|X^\prime|\cdot \epsilon n+(|X^{(i-1)+}|-|X^\prime|)\cdot|X^{(i-1)+}|&> \sum_{x\in X^{(i-1)+}}d^-(x, X^{(i-1)+})\\
&\geq\frac{|X^{(i-1)+}|(|X^{(i-1)+}|-2cn)}{2},
\end{split}
\end{equation*}
which implies that $|X^\prime|<\frac{|X^{(i-1)+}|(|X^{(i-1)+}|/2+cn)}{|X^{(i-1)+}|-\epsilon n}\leq \frac{|X^{(i-1)+}|}{2}+\epsilon n$. So $|X^{i+}|\geq\frac{|X^{(i-1)+}|}{2}-\epsilon n$.
Then $|X^{i+}|\geq\frac{|X|}{2^{i}}-\epsilon n\cdot\sum_{j=0}^{i-1}\frac{1}{2^j}\geq \frac{|X|}{2^{i}}-2\epsilon n$. Hence this proves the lemma.
%
\end{proof}
Lemmas \ref{l7}-\ref{l10} play a key role in the search for the $\delta$-extremal partition of $S$. 
\begin{lemma}\label{l7}\emph{(Li and Molla \cite{Li} - Lemma $20$)}
Suppose that $1/n\ll c\ll\gamma_2<\gamma_1/2, \gamma/2\ll1$, and that $D$ is an oriented graph on $n$ vertices with $\delta^0(D)\geq(1/2-c)n$. Let $A$ and $B$ be two disjoint vertex sets of $V(D)$ and $e^+(A, B)\leq \gamma_2 n^2$. Then $|N_{\gamma, B}^-(A)|\geq |B|-\gamma_1 n$ and $|N_{\gamma, A}^+(B)|\geq |A|-\gamma_1 n$.
\end{lemma}
\begin{lemma}\label{l14}
Suppose that
$1/n\ll c\ll \alpha\ll \gamma\ll\delta<\frac{1}{216k^{3}(k!)^2}$. Let $D$ be an oriented graph on $n$ vertices with $\delta^0(D)\geq(1/2-c)n$ and let $S\subseteq V(D)$ with $|S|\geq(1-6k^{3/2}\delta)n$.
If $A$ is a subset of $S$ with $|A|\geq\frac{n}{7k^{3/2}}$ and $k$-cyc$(A^{k-1}, S\setminus A)\leq \alpha n^k$, then the following hold.\\
$(i)$ $|N_{\gamma, S}^{+}(A)|$, $|N_{\gamma, S}^{-}(A)|=|S\setminus A|/2\pm50k^{3}\delta n$, and\\
$(ii)$ $|A|\leq(1/3+50k^{3}\delta)n$.
\end{lemma}
\begin{proof}
Suppose that 
$\alpha\ll\epsilon\ll\gamma_2\ll\gamma_1\ll \gamma\ll\eta<\frac{1}{216k^{3}(k!)^2}$. Let $S^+=N_{\gamma, S}^{+}(A)$, $S^-=N_{\gamma, S}^{-}(A)$ and $R=S\setminus(S^+\cup S^-\cup A)$. Set $V_0=V(D)\setminus S$. We first declare that
\begin{claim}\label{claim3.14}
$|R|\leq\gamma_1 n$.
\end{claim}
\begin{proof}
For any vertex $x$ in $R$, let $X=N^+(x, A)$ and $Y=N^-(x, A)$. Then by the definition of $R$, we have that $|X|, |Y|\leq |A|-\gamma n$. Since $d^+(x), d^-(x)\geq\delta^0(D)\geq(1/2-c)n$ and $|A|\geq\frac{n}{7k^{3/2}}$, it follows that
\begin{equation}\label{aaa4}
\begin{split}
|X|, |Y|\geq\gamma n-2cn\geq\gamma n/2\ \emph{\emph{and}}\ |X|+|Y|\geq|A|-2cn\geq4\eta n.
\end{split}
\end{equation}
Based on the number of arcs from $X$ to $Y$, we consider two cases as follows. 

\medskip

\noindent \textbf{Case 1.}
For any vertex $x\in R$, it satisfies that $e^+(X, Y)\geq\gamma_2 n^2$.

\smallskip

By the pigeonhole principle, there exist at least $|R|/2$ vertices $x$ such that $|X|\geq|Y|$ or $|X|\leq|Y|$. Without loss of generality, assume that there are $|R|/2$ vertices $x$ satisfying $|X|\geq|Y|$. Then by (\ref{aaa4}), we get that $|X|\geq(|A|-2cn)/2$. Let $R^\prime$ be the set of these $|R|/2$ vertices $x$.
In the following, for any vertex $x$ in $R^\prime$, we will calculate the lower bound of $k$-\emph{\emph{cyc}}($x, A^{k-1})$, which is an integer multiple of $|R|$. Together with $k$-cyc$(A^{k-1}, \overline{A})\leq \alpha n^k$, we can get $|R|\leq\gamma_1 n$ as requested.

\smallskip

Define $X^+=\{y\in X: d^+(y, Y)\geq\epsilon n\}$, and then
$\gamma_2 n^2\leq e^+(X, Y)<|X^+|\cdot|Y|+(|X|-|X^+|)\cdot\epsilon n,$
which suggests that $|X^+|\geq \gamma_2 n$ due to $|X|\geq|Y|$ and $\gamma_2\gg\epsilon$.
Recall that by Definition \ref{aaaaa} with $\sigma=+$, we have
$X^{i+}=\{x\in X^{(i-1)+}:\ d^-(x, X^{(i-1)+})\geq \epsilon n\}$, which implies that there are at least $\prod_{i=0}^{k-4}(\epsilon n-i)\cdot|X^{(k-2)+}|$ distinct $(k-2)$-paths from $X^+$ to $X^{(k-2)+}$.
Further, by Lemma \ref{lx8}, we get $|X^{i+}|\geq\frac{|X^{(i-1)+}|}{2}-\epsilon n$ and $|X^{(k-2)+}|\geq\frac{|X^+|}{2^{k-3}}-2\epsilon n$.
Then we obtain (see Figure \ref{c-1} (a))
\begin{center}
$k$-\emph{\emph{cyc}}($x, A^{k-1})\geq\prod_{i=0}^{k-4}(\epsilon n-i)|X^{(k-2)+}|\cdot \epsilon n\geq\epsilon^{k-2}\cdot\left(\frac{\gamma}{2^{k-3}}-3\epsilon\right)\cdot n^{k-1}$.
\end{center}
\begin{figure}[h]
\centering
\scriptsize
\bigskip
\begin{tabular}{ccc}
\includegraphics[width=6.3cm]{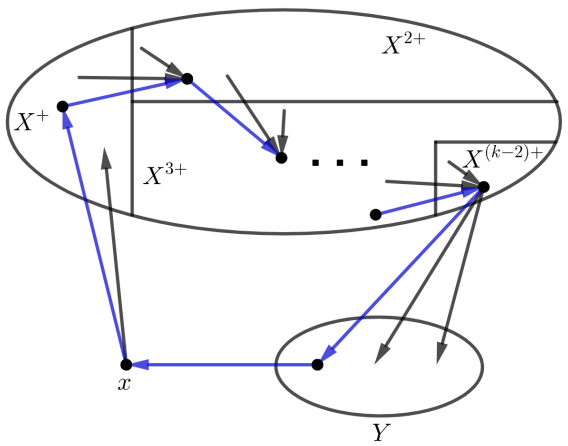}  \qquad\qquad&\includegraphics[width=6.3cm]{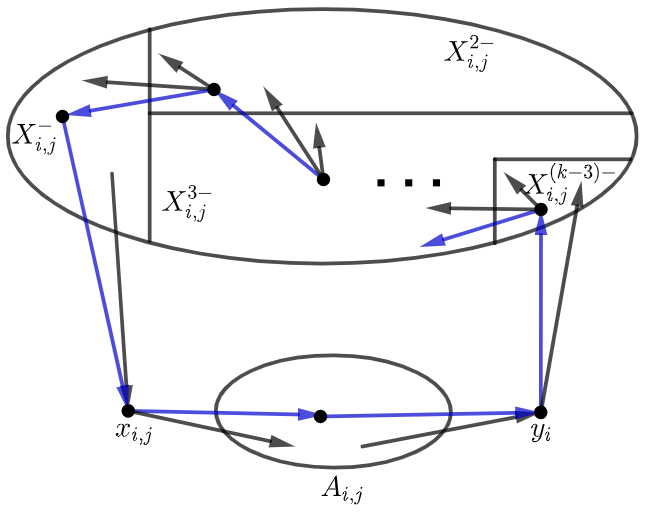}\\
 (a) Case $1$ when $e^+(X, Y)\geq\gamma_2 n^2$ for any vertex $x\in R$. \qquad\qquad &(b) Case $2$ when $e^+(X, Y)<\gamma_2 n^2$ for some vertex $x\in R$.
\end{tabular}

\caption{In the figure, the blue arcs form a $k$-cycle.}
\label{c-1}
\vspace{-0.5em}
\end{figure}
Furthermore, by $k$-cyc$(A^{k-1}, S\setminus A)\leq\alpha n^k$ and $A\cap R=\emptyset$, 
we have
\begin{equation*}
\begin{split}
\alpha n^k\geq k\emph{\emph{-cyc}}(S\setminus A, A^{k-1})&\geq k \emph{\emph{-cyc}}(R, A^{k-1})\geq\sum_{x\in R^\prime}k\emph{\emph{-cyc}}(x, A^{k-1})\\
&\geq\frac{|R|}{2}\cdot\epsilon^{k-2}\cdot\left(\frac{\gamma}{2^{k-3}}-3\epsilon\right)\cdot n^{k-1}.
\end{split}
\end{equation*}
By simplifying the above inequality and due to $\alpha\ll\epsilon\ll\gamma_2\ll\gamma_1$, we get that $|R|\leq\gamma_1 n$.

\medskip

\noindent \textbf{Case 2.}
There exists a vertex $x$ in $R$ such that $e^+(X, Y)<\gamma_2 n^2$.

\smallskip
Without loss of generality, we assume that $\epsilon n=\lfloor\epsilon n\rfloor$. In this case, we will show that there exist $\epsilon n$ distinct vertices $y_1, y_2, \ldots, y_{\epsilon n}$ in $Y$ and for any $y_i$ with $i\in[\epsilon n]$, there are $\epsilon n$ distinct vertices $x_{i, 1}, \ldots, x_{i, \epsilon n}$ in $X$ such that $k$-cyc$(S\setminus A, y_i, A^{k-3}, x_{i, j})\geq\frac{\eta\gamma\epsilon^{k-4}}{2^{k}}\cdot n^{k-2}$. 

Suppose that $\gamma^\prime$ is a real number with $\gamma_2\ll\gamma^\prime\ll\gamma$. By Lemma \ref{l7}, we have that
\begin{equation}\label{ab6}
\begin{split}
|N_{\gamma^\prime, Y}^-(X)|\geq|Y|-\gamma^\prime n\ \emph{\emph{and}}\ |N_{\gamma^\prime, X}^+(Y) |\geq|X|-\gamma^\prime n.
\end{split}
\end{equation}
This implies that there exists a subset $Y^\prime\subseteq Y$ with $|Y^\prime|\geq|Y|-\gamma^\prime n$ such that, for any vertex $y_i$ in $Y^\prime$, we have $d^+(y_i, X)\geq|X|-\gamma^\prime n$, and we define $X_i=N^+(y_i, X)$.
\medskip

\noindent \textbf{Step 2.1.}
Find $\epsilon n$ vertices $y_1, y_2, \ldots, y_{\epsilon n}$ in $Y^\prime$ such that for any $i\in[\epsilon n]$, we have that $d^-(y_i, S\setminus A)\geq\frac{|S\setminus A|+|X|}{2}-4k^{3/2}\delta n$.
\smallskip

To see this, for any $i\in[\epsilon n]$, by $\delta^0(D)\geq(1/2-c)n$ and Lemma \ref{lx2}-(ii), we can choose $y_i\in Y^\prime\setminus\{y_1, \ldots, y_{i-1}\}$ such that
$d^+(y_i, Y^\prime)\geq\frac{|Y^\prime|-(i-1)-2cn}{2}$. Then by (\ref{aaa4}), we have that $(1/2-c)n\leq d^-(y_i)\leq d^-(y_i, S\setminus A)+|V_0|+d^-(y_i, X)+d^-(y_i, Y)+2cn$. Since $d^-(y_i, X)\leq|X|-d^+(y_i, X)$, $d^-(y_i, Y)\leq(|Y|-|Y^\prime|)+|Y^\prime|-d^+(y_i, Y^\prime)$, $d^+(y_i, X)\geq|X|-\gamma^\prime n$ and $d^+(y_i, Y^\prime)\geq\frac{|Y^\prime|-(i-1)-2cn}{2}$, we obtain
$(1/2-c)n\leq d^-(y_i, S\setminus A)+|V_0|+\frac{|Y^\prime|}{2}+2\gamma^\prime n+3cn+\frac{i-1}{2}.$
Together with $|V_0|\leq6k^{3/2}\delta n$ and $Y^\prime\subseteq Y$, this yields that
$d^-(y_i, S\setminus A)
\geq\frac{|S\setminus A|+|X|}{2}-4k^{3/2}\delta n.$

\medskip
\noindent \textbf{Step 2.2.}
For any given $i\in[\epsilon n]$, find $\epsilon n$ vertices $x_{i, 1}, \ldots, x_{i, \epsilon n}$ in $X_i$ such that for any $j\in[\epsilon n]$, we have that $d^+(x_{i, j}, S\setminus A)\geq\frac{|S\setminus A|+|Y|}{2}-4k^{3/2}\delta n$.

\smallskip

Recall that by (\ref{aaa4}), $|X_i|\geq|X|-\gamma^\prime n\geq\gamma n/2-\gamma^\prime n>\gamma n/3$. Further, by (\ref{ab6}), for any given $i\in[\epsilon n]$, we can obtain that $|N_{\gamma^\prime}^+(Y)\cap X_i|\geq|X_i|-\gamma^\prime n$. Then by $\delta^0(D)\geq(1/2-c)n$ and Lemma \ref{lx2}-(ii) again, there is a vertex $x_{i, j}\in(N_{\gamma^\prime}^+(Y)\cap X_i)\setminus\{x_{i, 1}, \ldots, x_{i, j-1}\}$ with
\begin{equation}\label{aaa5}
\begin{split}
d^-(x_{i, j}, X_i)&\geq \frac{|N_{\gamma^\prime}^+(Y)\cap X_i|-(j-1)-2cn}{2}
\geq\frac{|X_i|-\gamma^\prime n-(j-1)-2cn}{2}.
\end{split}
\end{equation}
Because $x_{i, j}\in N_{\gamma^\prime}^+(Y)$, it is easily seen that
\begin{equation}\label{aaa6}
\begin{split}
d^-(x_{i, j}, Y)\geq|Y|-\gamma^\prime n.
\end{split}
\end{equation}
Hence by (\ref{aaa4}) and the lower bound of $\delta^0(D)$, we have that $(1/2-c)n\leq d^+(x_{i, j})\leq d^+(x_{i, j}, S\setminus A)+|V_0|+d^+(x_{i, j}, X)+d^+(x_{i, j}, Y)+2cn$. Also, due to (\ref{aaa5})-(\ref{aaa6}), $|V_0|\leq6k^{3/2}\delta n$, $d^+(x_{i, j}, X)\leq(|X|-|X_i|)+|X_i|-d^-(x_{i, j}, X_i)$, $d^+(x_{i, j}, Y)\leq|Y|-d^-(x_{i, j}, Y)$, and the property of vertex $x_{i, j}$, we obtain that
\begin{equation*}\label{ab2}
\begin{split}
d^+(x_{i, j}, S\setminus A)&\geq(1/2-c)n-6k^{3/2}\delta n-\left(\frac{|X_i|}{2}+\frac{5\gamma^\prime n}{2}+3cn+\frac{j-1}{2}\right)\\
&\geq\frac{|S\setminus A|+|Y|}{2}-4k^{3/2}\delta n.
\end{split}
\end{equation*}
So Step $2.2$ is completed.

Note that it follows from $x_{i, j}\in X_i$ that $y_ix_{i, j}\in A(D)$ for any $i, j\in[\epsilon n]$.
Set $A_{i, j}=N^+(x_{i, j}, S\setminus A)\cap N^-(y_i, S\setminus A)$ for any $i, j\in[\epsilon n]$. Then by Steps $2.1$-$2.2$ and (\ref{aaa4}) and $|A|\geq\frac{n}{7k^{3/2}}$, we have $|A_{i, j}|\geq\frac{|X|+|Y|}{2}-8k^{3/2}\delta n\geq\frac{|A|-2cn}{2}-8k^{3/2}\delta n\geq\eta n$. Let
$X_{i, j}^-=N^-(x_{i, j}, X_i)$ for any $i, j\in[\epsilon n]$. By (\ref{aaa5}), we get that $|X_{i, j}^-|\geq\frac{|X_i|-\gamma^\prime n-(j-1)-2cn}{2}$.
Also by Lemma \ref{lx8}-(ii), we have that $|X^{s-}_{i, j}|\geq\frac{|X^{(s-1)-}_{i, j}|}{2}-\epsilon n$ for any $s\in\{2, 3, \ldots, k-3\}$ and $|X^{(k-3)-}_{i, j}|\geq \frac{|X_{i, j}^-|}{2^{k-4}}-2\epsilon n$.
Recall that $X^{(k-3)-}_{i, j}\subseteq X_{i, j}^-=N^-(x_{i, j}, X_i)\subseteq X_i=N^+(y_i, X)$. Further, due to $|A_{i, j}|\geq\eta n$, $\alpha\ll\epsilon\ll\gamma\ll\eta$ and $k\emph{\emph{-cyc}}(S\setminus A, A^{k-1})\geq k\emph{\emph{-cyc}}(S\setminus A, y_i, A^{k-3}, x_{i, j})$, we have that (see Figure \ref{c-1} (b))
\begin{equation*}
\begin{split}
k\emph{\emph{-cyc}}(S\setminus A, A^{k-1})&\geq(\epsilon n)^2\cdot|A_{i, j}|\cdot(|X^{(k-3)-}_{i, j}|\cdot\prod_{i=0}^{k-5}(\epsilon n-i))\\
&\geq(\epsilon n)^2\cdot\eta n\cdot(\frac{|X_{i, j}^-|}{2^{k-4}}-2\epsilon n)\cdot\prod_{i=0}^{k-5}(\epsilon n-i)\nonumber\\
&\geq(\epsilon n)^2\cdot\frac{\eta\gamma\epsilon^{k-4}}{2^{k}}\cdot n^{k-2}>\alpha n^k.
\end{split}
\end{equation*}
This contradicts the fact that $k$-cyc$(A^{k-1}, S\setminus A)\leq \alpha n^k$. Hence Claim \ref{claim3.14} holds.
\end{proof}
In the following, we first prove (i) by using the fact that $|R|\leq\gamma_1 n$. According to the definition of $S^{-\sigma}$ for $\sigma\in\{+, -\}$ (i.e., $S^{-\sigma}=S^-$ if $\sigma=+$, and $S^{-\sigma}=S^+$ if $\sigma=-$), we have that
\begin{center}
$e^\sigma(A, S\setminus A)\leq|A|\cdot|S^\sigma|+(|A|-\gamma n)\cdot|R|+\gamma n\cdot|S^{-\sigma}|\leq|A|\cdot|S^\sigma|+2\gamma n^2$.
\end{center}
So $e^\sigma((A, S\setminus A)=|A|\cdot|S^\sigma|\pm2\gamma n^2$. Together with Lemma \ref{lx2}-(i), we get that $e^+(A, \overline{A})=e^{-}(A, \overline{A})\pm2\gamma n^2$. Since $\overline{A}=(S\setminus A)\cup V_0$ and $|V_0|\leq6k^{3/2}\delta n$, we have $e^+(A, S\setminus A)=e^-(A, S\setminus A)\pm 7k^{3/2}\delta n$. It follows from $e^\sigma(A, S\setminus A)=|A|\cdot|S^\sigma|\pm2\gamma n^2$ that
\begin{equation*}
\begin{split}
|A|\cdot|S^+|=e^{-}(A, S\setminus A)\pm(2\gamma+7k^{3/2}\delta )n^2 =|A|\cdot|S^-|\pm(2\gamma+14k^{3/2}\delta)n^2,
\end{split}
\end{equation*}
Therefore, because $|A|\geq \frac{n}{7k^{3/2}}$, it is not hard to check that $|S^+|=|S^-|\pm99k^{3}\delta n$. In addition, together with Claim \ref{claim3.14}, we obtain
\begin{center}
$|S\setminus A|=|S^+|+|S^-|+|R|=2|S^\sigma|\pm100k^{3}\delta n$, for any $\sigma\in\{+, -\}$.
\end{center}
Hence Lemma \ref{l14}-(i) is proved.

\smallskip

Next, we will give the proof of (ii). Suppose, to the contrary, that $|A|>(1/3+50k^{3}\delta)n$, which implies $|S\setminus A|\leq|\overline{A}|<(2/3-50k^{3}\delta)n$. By (i), we have that $|S^+|\geq\frac{|S\setminus A|}{2}-50k^{3}\delta n$. Clearly, there exists a vertex $x\in S^+$ satisfying $d^+(x, S^+)\leq|S^+|/2$. Then $d^+(x, A)\leq\gamma n$ due to $x\in S^+$. Further, combining with $d^+(x)=d^{+}(x, \overline{A})+d^+(x, A)$, $d^{+}(x, \overline{A})\leq|S^-|+|R|+|V_0|+|S^+|/2\leq|S\setminus A|-|S^+|/2+6k^{3/2}\delta n$, and $c\ll\delta$, we have that
\begin{equation*}
\begin{split}
d^+(x)&\leq|S\setminus A|-\frac{|S\setminus A|}{4}+25k^{3}\delta n+6k^{3/2}\delta n+\gamma n\\
&<\frac{3}{4}\cdot(2/3-50k^{3}\delta)n+32k^{3}\delta n<(1/2-c)n.
\end{split}
\end{equation*}
This contradicts the fact that $\delta^0(D)\geq(1/2-c)n$, and so Lemma \ref{l14}-(ii) is true.
Thus, we complete the proof of Lemma \ref{l14}.
\end{proof}
\begin{lemma}\label{l10}
Assume that $1/n\ll\alpha, c\ll\gamma\ll\eta, \delta<\frac{1}{216k^{3}(k!)^2}$. Let $D$ be an oriented graph on $n$ vertices with $\delta^0(D)\geq(1/2-c)n$ and let $S\subseteq V(D)$ with $|S|\geq(1-6k^{3/2}\delta)n$. For $A\subseteq S$ such that $|A|\geq \eta n$ and $k$-cyc$(A)\leq \alpha n^k$, the following statements hold. \\
$(i)$ For $\sigma\in\{-, +\}$, there is $x^{\sigma}\in A$ satisfying $d^{\sigma}(x^{\sigma}, A)\geq |A|-\gamma n$.\\
$(ii)$ For every partition $\{A, B, C\}$ of $S$ with $k$-cyc$(A^{k-1}, C)\leq \alpha n^k$, there are disjoint subsets $C^+, C^-\subseteq C$ such that for $\sigma\in\{+, -\}$,
\begin{center}
$(\star)$ $e^{-\sigma}(A, C^\sigma)\leq\gamma n^2$ and $|C^\sigma|\geq|C|/2+(|A|-|B|)/2-7k^{3/2}\delta n$.
\end{center}
\end{lemma}
\begin{proof}
Suppose that $\epsilon, \gamma_1$, $\gamma_2$ and $\gamma_3$ be real numbers with $\alpha\ll\epsilon, \gamma_3\ll\gamma_1, \gamma_2\ll\gamma$.

We first prove (i). It suffices to prove that there exists a vertex $x^-\in A$ satisfying $d^-(x^-, A)\geq|A|-\gamma n$, and the existence of the vertex $x^+$ is a similar argument.
Let $X$ be the set of vertices $x$ in $A$ satisfying $k$-cyc($x, A^{k-1}$)$\leq \alpha^{1/2}n^{k-1}$. Since
$(|A|-|X|)\cdot\alpha^{1/2}n^{k-1}\leq\sum_{x\in X\setminus A}k$-cyc($x, A^{k-1}$)$\leq k\cdot k$-cyc($A$)$\leq \alpha kn^k$,
we have that
\begin{equation}\label{e1}
\begin{split}
|X|\geq|A|-\alpha^{1/2}kn.
\end{split}
\end{equation}
Pick $x^{-}\in X$ with $d^{-}(x^-, X)$ maximal. Set $X^+=N^+(x^-, X)$ and $X^-=N^-(x^-, X)$. In the following, we will show that $|X^-|\geq|A|-\gamma n$, which indicates that $x^-$ is the vertex with $d^-(x^-, A)\geq|A|-\gamma n$ as required.

According to $\delta^0(D)\geq(1/2-c)n$, it is not hard to see that $|X^+|\geq |X|-|X^-|-2cn$.
On the one hand, due to $x^-\in X$, we have $k$-cyc$(x^-, A^{k-1})\leq\alpha^{1/2}n^{k-1}$. On the other hand, we can estimate the lower bound of $k$-cyc$(x^-, A^{k-1})$. Recall that by Definition \ref{aaaaa}, we have $X^{i+}=\{x\in X^{(i-1)+}:\ d^-(x, X^{(i-1)+})\geq \epsilon n\}$. This implies that for any $x_{k-2}\in X^{(k-2)+}$, we obtain that $d^-(x_{k-2}, X^{(k-3)+})\geq\epsilon n$. Similarly, by the definition of $X^{(k-3)+}$, for any $x_{k-3}\in X^{(k-3)+}\cap N^-(x_{k-2}, X^{(k-3)+})$, we have that $d^-(x_{k-3}, X^{(k-4)+})\geq\epsilon n$. Repeating the process above, this gives at least $\prod_{i=0}^{k-4}(\epsilon n-i)$ distinct $(k-2)$-paths from $X^+$ to $X^{(k-2)+}$. So we get that $k$-\emph{\emph{cyc}}($x^-, A^{k-1})\geq e^{+}(X^{(k-2)+}, X^-)\cdot\prod_{i=0}^{k-4}(\epsilon n-i)$. Hence
\begin{center}
$\alpha^{1/2}n^{k-1}\geq$ $k$-\emph{\emph{cyc}}($x^-, A^{k-1})\geq e^{+}(X^{(k-2)+}, X^-)\cdot\prod_{i=0}^{k-4}(\epsilon n-i)$,
\end{center}
which implies that $e^{+}(X^{(k-2)+}, X^-)\leq \frac{2}{\epsilon^{k-3}}\cdot\alpha^{1/2}n^2\leq\gamma_3 n^2$.
It follows from Lemmas \ref{lx8}-\ref{l7} that
\begin{equation*}
\begin{split}
|N_{\gamma_1}^{+}(X^-)\cap X^{(k-2)+}|\geq|X^{(k-2)+}|-\gamma_2 n\geq\frac{|X^{+}|}{2^{k-3}}-2\epsilon n-\gamma_2 n\geq\frac{|X^{+}|}{2^{k-3}}-2\gamma_2 n.
\end{split}
\end{equation*}

Further, by $\delta^0(D)\geq(1/2-c)n$ and the pigeonhole principle, there exists $w\in N_{\gamma_1}^{+}(X^-)\cap X^{(k-2)+}$ such that $d^-(w, X^{(k-2)+})\geq \frac{|N_{\gamma_1}^{+}(X^-)\cap X^{(k-2)+}|-2cn}{2}\geq\frac{|X^{+}|}{2^{k-2}}-2\gamma_2 n$.
Since $w\in X^{(k-2)+}\subseteq X^+\subseteq X$, and $x^-$ has the maximum in-degree in $X$, and $d^{-}(w, X^-)\geq|X^-|-\gamma_1 n$ due to $w\in N_{\gamma_1}^{+}(X^-)$, we obtain that
\begin{equation*}
\begin{split}
|X^-|\geq d^-(w, X^{(k-2)+})+d^{-}(w, X^-)
\geq\frac{|X|-|X^-|-2cn}{2^{k-2}}-2\gamma_2 n+|X^-|-\gamma_1 n.
\end{split}
\end{equation*}
This implies that $|X^-|\geq|X|-2cn-2^{k-2}(2\gamma_2 +\gamma_1)n$. Combining with (\ref{e1}) and $c\ll\gamma_1, \gamma_2\ll\gamma$, we deduce that $d^-(x^-, A)\geq|X^-|\geq|A|-\gamma n$.

\smallskip

We secondly prove (ii). It will be proved with $\gamma/2$ taking on the role of $\gamma$. In particular, if there are subsets $C^+$ and $C^-$ satisfying $(\star)$ but $C^+\cap C^-$ is not empty, then
\begin{center}
$(|A|-2cn)\cdot|C^+\cap C^-|\leq e^+(A, C^+\cap C^-)+e^-(A, C^+\cap C^-)\leq\gamma n^2$.
\end{center}
Together with $|A|\geq \eta n$, this means $|C^+\cap C^-|\leq \gamma n/2$. Therefore, in this case, the disjoint vertex sets $C^+\setminus C^-$ and $C^-\setminus C^+$ satisfy $(\star)$.

In the following, we will only prove the existence of the vertex set $C^+$ satisfying ($\star$). Using the similar derivation, we can also prove that the desired vertex set $C^-$  exists. 

Let $Y$ be the set of vertices $y$ in $A$ satisfying $k$-cyc($y, A^{k-2}, C$)$\leq \alpha^{1/2}n^{k-1}$. Then by $k$-cyc$(A^{k-1}, C)\leq \alpha n^k$, we have
\begin{center}
$(|A|-|Y|)\cdot\alpha^{1/2}n^{k-1}\leq \sum_{y\in A\setminus X}k$-cyc($y, A^{k-2}, C)\leq k\cdot k$-cyc($A^{k-1}, C$)$\leq \alpha kn^k$.
\end{center}
This implies that $|Y|\geq|A|-\alpha^{1/2}kn$.
Due to $k$-cyc($Y$)$\leq k$-cyc($A$)$\leq \alpha n^k$, if follows from (i) that there is $y\in Y$ satisfying
\begin{equation*}
\begin{split}
d^-(y, A)\geq d^-(y, Y)\geq|Y|-\gamma n/4\geq|A|-\gamma n/3.
\end{split}
\end{equation*}
Let $C^+=N^+(y, C)=\{y_1, \ldots, y_{|C^+|}\}$, $A_1=N^-(y, A)$ and $Y_j=N^+(y_j, A_1)$ for any $j\in[|C^+|]$. Clearly $|A\setminus A_1|\leq\gamma n/3$. Next, we will give an upper bound of $e^{+}(C^+, A)$ by estimating the upper and lower bounds of $k$-cyc($y, A^{k-2}, C^+$).

We first count $k$-cyc$(y, y_i, A^{k-2})$ for any $y_i\in C^+$. By Definition \ref{aaaaa}-(i), we know $Y_j^{i+}=\{y\in Y^{(i-1)+}_j: d^-(y, Y_j^{(i-1)+})\geq\epsilon n\}$, which gives at least $\prod_{s=0}^{k-4}(\epsilon n-s)$ distinct $(k-3)$-paths from $Y_j$ to $Y_{j}^{(k-4)+}$. Also, for any $j\in[|C^+|]$, by Lemma \ref{lx8}-(i), we have that $|Y_j^{i+}|\geq\frac{|Y^{(i-1)+}|}{2}-\epsilon n$ for any $i\in[k-3]$ and $|Y_j^{(k-3)+}|\geq\frac{|Y_j|}{2^{k-3}}-2\epsilon n$. Hence
\begin{equation*}
\begin{split}
k\emph{\emph{-cyc}}(y, y_j, A^{k-2})&\geq(|Y_j^{(k-3)+}|-(k-3))\cdot\prod_{s=0}^{k-4}(\epsilon n-s)\\
&\geq\left(\frac{|Y_j|}{2^{k-3}}-2\epsilon n-(k-3)\right)\cdot\prod_{s=0}^{k-4}(\epsilon n-s).
\end{split}
\end{equation*}
Furthermore, we have that
\begin{equation*}
\begin{split}
k\emph{\emph{-cyc}}(y, C^+, A^{k-2})=\sum_{j=1}^{|C^+|}k\emph{\emph{-cyc}}(y, y_j, A^{k-2})\geq(\sum_{j=1}^{|C^+|}\frac{|Y_j|}{2^{k-3}}-(2\epsilon n+k-3)|C^+|)\prod_{s=0}^{k-4}(\epsilon n-s).
\end{split}
\end{equation*}
In addition, we deduce that $$\sum_{j=1}^{|C^+|}|Y_j|=e^+(C^+, A_1)=e^+(C^+, A)-e^+(C^+, A\setminus A_1)\geq e^+(C^+, A)-\gamma n^2/3.$$
Together with $k$-cyc$(y, A^{k-2}, C^+)\leq k$-cyc$(y, A^{k-2}, C)\leq\alpha^{1/2}n^{k-1}$ since $y\in Y$, we can obtain that
\begin{align}\label{wzg1}
\alpha^{1/2}n^{k-1}\geq k\emph{\emph{-cyc}}(y, C^+, A^{k-2})
&\geq\left(\frac{e^+(C^+, A)-\gamma n^2/3}{2^{k-3}}-(2\epsilon n+k-3)|C^+|\right)\prod_{s=0}^{k-4}(\epsilon n-s)\nonumber\\
&\geq\left(\frac{e^+(C^+, A)}{2^{2k-6}}-\frac{\gamma n^2}{2^{2k-4}}\right)\epsilon^{k-3} n^{k-3}.
\end{align}
Hence, due to $\alpha\ll\epsilon\ll\gamma$, this implies $e^-(A, C^+)=e^{+}(C^+, A)\leq\gamma n^{2}/2$.

Finally, since $V(D)=A\cup B\cup C\cup V_0$, $|C^+|\geq d^+(y)-d^+(y, A)-|B|-|V_0|$, and $|V_0|\leq6k^{3/2}\delta n$, and $d^+(y, A)=|A|-d^-(y, A)\leq \gamma n/3$, we also have that
\begin{equation*}
\begin{split}
|C^+|\geq|C|/2+(|A|-|B|)/2-\gamma n/2-6k^{3/2}\delta n\geq|C|/2+(|A|-|B|)/2-7k^{3/2}\delta n,
\end{split}
\end{equation*}
since $d^+(y)\geq\delta^0(D)\geq (1/2-c)n$ and $n\geq |A|+|B|+|C|$. Therefore, the lemma holds.
\end{proof}


\section{The proof of Theorem \ref{t1}}
Choose a constant $c^\prime$ with $1/2>c^\prime>\frac{1}{6k^{1/2}k!}$ such that Theorem \ref{KKK} holds. 
Let $n_0$, $n$, $l$ and $d$ be positive integers with $n\geq n_0$, and suppose that $c\leq c^\prime-\frac{1}{6k^{1/2}k!}$ and
\begin{flushleft}
$1/n\ll c\ll\beta^\prime\ll\beta\ll\mu\ll\alpha^\prime\ll\epsilon\ll \alpha\ll\gamma_3\ll\gamma_2\ll\gamma_1\ll$
\end{flushleft}
\begin{flushright}
$\gamma\ll\gamma^\prime\ll\eta, \delta<1/(216k^{3}(k!)^2)<\delta_1<1/ (20k^2), 1/ d, 1/l$
\end{flushright}
Let $D$ be an oriented graph on $n$ vertices with $\delta^0(D)\geq(1/2-c)n$, and let $H(D)$ be the $k$-uniform hypergraph of $D$.
For each $v\in V(D)$, recall that $\widetilde{N}_{H(D), \alpha, 1}(v)$ is the set of vertices that are $(H(D), \alpha, 1)$-reachable to $v$. Together with Lemmas \ref{l3}-\ref{l12} in which we replace $f$ with $6k^{3/2}$, we get that there exists $S\subseteq V(D)$ with $|S|\geq(1-6k^{3/2}\delta)n$ such that $|\widetilde{N}_{H(D), \alpha, 1}(v)\cap S|\geq\delta n$ for any $v\in S$. Then by Lemma \ref{l13}, we can find a partition of $S$ into $V_1, \ldots, V_r$ with $r\leq\min\{6k^{3/2}, \lfloor1/\delta\rfloor\}$ such that $|V_i|\geq(\delta-\alpha)n>\eta n$ and $V_i$ is $(H(D), \beta, 2^{6k^{3/2}-1})$-closed in $H(D)$ for any $i\in[r]$. Let $V_0=V(D)\setminus S$ and clearly $|V_0|\leq6k^{3/2}\delta n$. Set $\mathcal{P}^\prime=\{V_0, V_1, \ldots, V_r\}$.
\begin{claim}\label{songzhi}
If $S$ is $(H(D), \beta, l)$-closed, then Theorem \ref{t1} holds.
\end{claim}
\begin{proof}
By Lemma \ref{l3}, we have that $H(D)$ satisfies Lemma \ref{l5} with $\frac{1}{6k^{1/2}k!}$ playing the role of $\eta$.
Let $U$ be the set guaranteed by
Lemma \ref{l5} and $D^\prime=D-U$. Notice that
$$\delta^0(D^\prime)\geq\delta^0(D)-U\geq(1/2-c)n-\frac{n}{6k^{1/2}k!}\geq(1/2-c^\prime)|D^\prime|.$$
Thus Theorem \ref{KKK} implies that there exists a constant $C>0$ such that $D^\prime$ contains disjoint $k$-cycles covering all but at most $C$ vertices of $D^\prime$. 
Let $W$ be the vertex set without being covered. Clearly, $|W|\leq C\leq \varepsilon |D^\prime|\leq\varepsilon n$ and $|W|\equiv0$ (mod $k$) due to $|V(D^\prime)|\equiv0$ (mod $k$). Again using Lemma \ref{l5}, we able to deduce that there is a perfect matching $\mathcal{M^\prime}$ in $H(D)[U\cup W]$, which implies there is a $k$-cycle-factor in $D[W\cup U]$. Therefore, it can be concluded that $D$ has a $k$-cycle-factor, which completes the proof of Claim \ref{songzhi}.
\end{proof}
In the following, we will only treat the case when $S$ is not ($H(D), \beta, l$)-closed. If $L_{\mathcal{P}^\prime}^\mu(H(D))$ contains a $2$-transferral $\mathbf{u}_i-\mathbf{u}_j$ for distinct $i, j\in[r]$, then we merge the partite sets $V_i$ and $V_j$ and consider the new partition $\mathcal{P}^\prime-V_i-V_j+(V_i\cup V_j)$ of $V(D)$. By Lemma \ref{l6}, we continue to merge the partite sets corresponding to $2$-transferrals until we get a final partition $\mathcal{P}=\{V_0, V_1, \ldots, V_d\}$ of $V(D)$ such that $\mathcal{P}$ contains no $2$-transferral $\mathbf{u}_i-\mathbf{u}_j$ for distinct $i, j\in[d]$, and for some integer $l^\prime$ with $l^\prime\geq2^{6k^{3/2}-1}$, $\mathcal{P}_1\setminus V_0$ is a $(H(D), \beta^\prime, l^\prime)$-closed partition of $S$ with $|V_i|\geq \eta n$ for every $i\in[d]$. For convenience, let $\mathcal{P}_1=\mathcal{P}\setminus V_0$. If $|\mathcal{P}_1|=1$, then $S$ is $(H(D), \beta, l)$-closed which contradicts the assumptions. Therefore, we assume that $|\mathcal{P}_1|\geq2$.
\subsection{Properties of the partition $\mathcal{P}_1$}
In this part, we present two claims that highlight the favorable properties of the partition $\mathcal{P}_1$ of $S$ and are frequently used in proving Theorem \ref{t1}. For simplicity, we write $D[S]$ as $S$ throughout the following discussion.
\begin{claim}\label{a1}
There is a partite set $A\in\mathcal{P}_1$ with $|A|\geq \frac{n}{7k^{3/2}}$ such that $k$-cyc$(A^{k-1}, S\setminus A)\leq\alpha n^k$.
\end{claim}
\begin{proof}
For a contradiction, suppose that
\begin{center}\label{cen2}
{\boldmath$(O1)$} $k$-cyc$(V_i^{k-1}, S\setminus V_i)\geq \alpha n^k$\qquad for all\ $V_i\in \mathcal{P}_1=\{V_1, \ldots, V_d\}$.
\end{center}
Because $|V_i|\geq \eta n$ for every $i\in[d]$, we get $\eta\leq1/d$. This means that
\begin{center}\label{cen3}
{\boldmath$(O2)$} for any $J\subseteq[d]$ and every $i\in[d]$, there exists $j^\prime\in J$ with $k$-cyc$(V_i^{k-1}, V_{j^\prime})\geq \eta\cdot k$-cyc$(V_i^{k-1}, \bigcup_{j\in J}V_j)$.
\end{center}
Let $A=V_x$ be the largest part of $\mathcal{P}_1$. It follows by $d\leq6k^{3/2}$ that $|A|\geq \frac{n-|V_0|}{6k^{3/2}}\geq \frac{n}{7k^{3/2}}.$ By ($O1$) and ($O2$) with $J=[d]-x$, there is $V_y=B\in \mathcal{P}_1-A$ such that $k$-cyc$(A^{k-1}, B)\geq\alpha^\prime n^k$. Let $C=S\setminus(A\cup B)$. Similarly, by ($O1$) and ($O2$) with $J=[d]-y$, there is $V_z=Q\in\mathcal{P}_1-B$ such that $k$-cyc$(B^{k-1}, Q)\geq\alpha^\prime n^k$. Let $F=S\setminus(B\cup Q)$. In the following, we first declare 
\begin{center}\label{OO3}
{\boldmath$(O3)$} $k$-cyc$(A)$, $k$-cyc$(A^{k-1}, C)$, $k$-cyc$(B)$, $k$-cyc$(B^{k-1}, F)$$<\alpha^\prime n^k$.
\end{center}

We only prove $k$-cyc$(A)$, $k$-cyc$(A^{k-1}, C)$$<\alpha^\prime n^k$, since the upper bounds of $k$-cyc$(B)$ and $k$-cyc$(B^{k-1}, F)$ can be similarly proved.

It is obvious that $\mathbf{v}_1=(k-1)\mathbf{u}_x+\mathbf{u}_y$ and $\mathbf{v}_2=(k-1)\mathbf{u}_y+\mathbf{u}_z$ belong to $I_{\mathcal{P}_1}^{\mu}(H(S))$. If $k$-cyc$(A)\geq\alpha^\prime n^k$, then $\mathbf{v}_3=k\mathbf{u}_x\in I_{\mathcal{P}_1}^{\mu}(H(S))$. Hence
$\mathbf{v}_1-\mathbf{v}_3$ is a $2$-transferral in $L_{\mathcal{P}}^{\mu}(H(S))$. Also, if $k$-cyc$(A^{k-1}, C)\geq\alpha^\prime n^k$, then by ($O2$) with $J=[d]-x-y$, there is $V_w\in \mathcal{P}_1-A-B$ such that $k$-cyc$(A^{k-1}, V_w)\geq \alpha^\prime n^k$. This implies that $\mathbf{v}_4=(k-1)\mathbf{u}_x+\mathbf{u}_w\in I_{\mathcal{P}_1}^{\mu}(H(S))$. So $\mathbf{v}_1-\mathbf{v}_4$ is a $2$-transferral in $L_{\mathcal{P}_1}^{\mu}(H(S))$. With the fact that $L_{\mathcal{P}_1}^{\mu}(H(S))$ is $2$-transferral-free, these imply that ($O3$) holds.

By Lemma \ref{l10} and ($O3$), there are disjoint subsets $C^+$ and $C^-$ of $C$ such that, for $\sigma\in \{-, +\}$ \begin{equation}\label{m1}
\begin{split}
e^{-\sigma}(A, C^\sigma)\leq\gamma n^2\ \mbox{and}\ |C^\sigma|\geq|C|/2+(|A|-|B|)/2-7k^{3/2}\delta n.
\end{split}
\end{equation}
%
Through the choice of $A$, we have that $|A|\geq|B|$. Hence (\ref{m1}) suggests that
\begin{equation}\label{m3}
\begin{split}
|C^\sigma|\geq|C|/2-7k^{3/2}\delta n.
\end{split}
\end{equation}
Since $C^+$ and $C^-$ are disjoint, we get that $\min\{|C^+|, |C^-|\}\leq|C|/2$. Together with (\ref{m1}), we have that
\begin{equation}\label{m4}
\begin{split}
|B|\geq|A|-14k^{3/2}\delta n.
\end{split}
\end{equation}
Analogously, By Lemma \ref{l10} again, with $B$ and $F$ taking on the roles of $A$ and $C$ respectively, there exist disjoint subsets $F^+, F^-\subseteq F$ such that, for $\sigma\in \{-, +\}$
\begin{equation}\label{m5}
\begin{split}
e^{-\sigma}(B, F^\sigma)\leq\gamma n^2 \ \mbox{and}\ |F^\sigma|\geq|F|/2+(|B|-|Q|)/2-7k^{3/2}\delta n.
\end{split}
\end{equation}
%
By the choice of $A$ and (\ref{m4}), it is obtained that $|B|+14k^{3/2}\delta n\geq|A|\geq|Q|$. So $|F^\sigma|\geq|F|/2-14k^{3/2}\delta n$, and
\begin{equation}\label{m6}
\begin{split}
|F^+|+|F^-|\geq|F|-28k^{3/2}\delta n.
\end{split}
\end{equation}
As follows, we will think about two cases based on whether sets $A$ and $Q$ are same.

\medskip
\begin{Case}\label{case3}
$A\neq Q$.
\end{Case}
In this case, we will search for a vertex subset $B^\prime\subseteq B$ such that there is a vertex $x\in B^\prime$ with $d^+(x)>(1/2+c)n$, which will imply that $d^-(x)<(1/2-c)n$, a contradiction. 

Since $A\neq Q$, we have that $A\subseteq F$. Fix $\sigma^\prime\in\{-, +\}$ satisfying $|F^{-\sigma^\prime}\cap A|\geq|F^{\sigma^\prime}\cap A|$. Without loss of generality, assume $\sigma^\prime=+$. Then by (\ref{m6}), we have
\begin{equation}\label{m7}
\begin{split}
|F^-\cap A|\geq(|F^-\cap A|+|F^+\cap A|)/2\geq(|A|-28k^{3/2}\delta n)/2=|A|/2-14k^{3/2}\delta n.
\end{split}
\end{equation}
By (\ref{m1}) and Lemma \ref{l7}, we get that $|A\setminus N_{\gamma^\prime}^-(C^+)|\leq \frac{\gamma^\prime n}{2}$. It follows by Lemma \ref{l7}, (\ref{m5}) and (\ref{m7}) that $|N_{\gamma^\prime}^-(B)\cap F^-\cap A|\geq|A|/2-15k^{3/2}\delta n$. Let $A^\prime=N_{\gamma^\prime}^-(B)\cap N_{\gamma^\prime}^-(C^+)\cap A$, then we have that
\begin{equation*}
\begin{split}
|A^\prime|\geq|N_{\gamma^\prime}^-(B)\cap F^-\cap A|-|A\setminus N_{\gamma^\prime}^-(C^+)|\geq|A|/2-16k^{3/2}\delta n.
\end{split}
\end{equation*}
Since $k$-cyc$(A^\prime)\leq k$-cyc$(A^k)\leq\alpha n^k$, by Lemma \ref{l10}-(i), there is $x\in A^\prime$ such that
\begin{equation*}
\begin{split}
d^{+}(x, A^\prime)\geq|A^\prime|-\gamma^\prime n\geq|A|/2-17k^{3/2}\delta n.
\end{split}
\end{equation*}
Clearly $x\in N_{\gamma^\prime}^-(B)\cap N_{\gamma^\prime}^-(C^+)$, which suggests that $d^+(x, B)\geq|B|-\gamma^\prime n$ and
\begin{center}
$d^+(x, C^+)\geq|C^+|-\gamma^\prime n\geq|C|/2-7k^{3/2}\delta n-\gamma^\prime n\geq|C|/2-8k^{3/2}\delta n$.
\end{center}
Hence, combining with (\ref{m4}), $|A|\geq\frac{n}{7k^{3/2}}$ and $ c\ll\delta<1/(216k^{3}(k!)^2)$, we obtain that
\begin{equation*}
\begin{split}
d^+(x)&\geq d^+(x, A^\prime)+d^+(x, B)+d^+(x, C^+)\\
&\geq(|A|/2-17k^{3/2}\delta n)+(|B|-\gamma^\prime n)+(|C|/2-8k^{3/2}\delta n)\\
&=n/2+|B|/2-\gamma^\prime n-25k^{3/2}\delta n\\
&\geq n/2+|A|/2-33k^{3/2}\delta n>(1/2+c)n.
\end{split}
\end{equation*}
So, $d^-(x)<(1/2-c)n$, which contracts with the fact that $\delta^0(D)\geq(1/2-c)n$.

\medskip
\begin{Case}\label{case4}
$A=Q$.
\end{Case}
In this case, we will show that either $k$-cyc($(F^+\cap C^+)^{k-2}, F^-\cap C^-, A$)$\geq\alpha n^k$ and $k$-cyc($(F^+\cap C^+)^{k-2}, F^-\cap C^-, B$)$\geq\alpha n^k$ in Subcase \ref{subcase1}, or $k$-cyc($A^{k-2}, B, R$)$\geq\mu n^k$ for some $R\in \mathcal{P}_1\setminus \{A, B\}$ and $k$-cyc($A^{k-2}, B, A$)$\geq\alpha^\prime n^k$ in Subcase \ref{subcase2}. This will lead to a 2-transferral in $L_{\mathcal{P}_1}^\mu(H(S))$ in each subcase, a contradiction. 

By (\ref{m1}) and Lemma \ref{l7}, we have that
\begin{equation}\label{a4}
\begin{split}
&|N_{\gamma^\prime}^-(C^+)\cap A|\geq|A|-\gamma^\prime n/2,\ |N_{\gamma^\prime}^+(A)\cap C^+|\geq|C^+|-\gamma^\prime n/2,\\
&|N_{\gamma^\prime}^+(C^-)\cap A|\geq|A|-\gamma^\prime n/2,\ |N_{\gamma^\prime}^-(A)\cap C^-|\geq|C^-|-\gamma^\prime n/2.\\
\end{split}
\end{equation}
Also, by (\ref{m5}) and Lemma \ref{l7}, we get that
\begin{equation}\label{a5}
\begin{split}
&|N_{\gamma^\prime}^-(F^+)\cap B|\geq|B|-\gamma^\prime n/2,\ |N_{\gamma^\prime}^+(B)\cap F^+|\geq|F^+|-\gamma^\prime n/2,\\
&|N_{\gamma^\prime}^+(F^-)\cap B|\geq|B|-\gamma^\prime n/2,\ |N_{\gamma^\prime}^-(B)\cap F^-|\geq|F^-|-\gamma^\prime n/2.\\
\end{split}
\end{equation}
\begin{Subcase}\label{subcase1}
$|F^+\cap C^-|\leq 2\gamma^\prime n$ or $|F^-\cap C^+|\leq2\gamma^\prime n$.
\end{Subcase}
Without loss of generality, assume that $|F^+\cap C^-|\leq 2\gamma^\prime n$. Together with $|F^\sigma|\geq|F|/2-14k^{3/2}\delta n=|C|/2-14k^{3/2}\delta n$ with $\sigma\in\{-, +\}$ and $|C^{-\sigma}|\leq|C|/2+7k^{3/2}\delta n$ by (\ref{m3}), we obtain that
\begin{align}\label{a6}
|F^+\cap C^+|&=|F^+\cap C|-|F^+\cap C^-|-|F^+\cap(C\setminus\{C^+, C^-\})|\nonumber\\
&\geq(|C|/2-14k^{3/2}\delta n)-2\gamma^\prime n-14k^{3/2}\delta n\geq|C^+|-36k^{3/2}\delta n.
\end{align}
This implies $|F^-\cap C^+|\leq36k^{3/2}\delta n$ since $F^+$ and $F^-$ are disjoint, and then
\begin{align}\label{a7}
|F^-\cap C^-|&=|F^-\cap C|-|F^-\cap C^+|-|F^-\cap(C\setminus\{C^+, C^-\})|\nonumber\\
&\geq(|C|/2-14k^{3/2}\delta n)-36k^{3/2}\delta n-14k^{3/2}\delta n\geq|C^-|-71k^{3/2}\delta n.
\end{align}
According to Lemma \ref{a3}, we have that $D[F^+\cap C^+]$ contains at least $|F^+\cap C^+|/2$ vertices of out-degree at most $3|F^+\cap C^+|/4+cn$ in $D[F^+\cap C^+]$. For convenience, we denote the set of these vertices by $X$. Let $Y=N_{\gamma^\prime}^+(A)\cap N_{\gamma^\prime}^+(B)\cap X$ and then $|Y|\geq|X|-\gamma^\prime n$ by (\ref{a4}) and (\ref{a5}). By definitions of $F^+$, $C^+$ and $Y$, we can infer that for any $y\in Y$,
\begin{align}\label{a8}
d^+(y, F^-\cap C^-)&\geq\delta^0(D)-d^+(y, F^+\cap C^+)-d^+(y, A\cup B)-127k^{3/2}\delta n\nonumber\\
&\geq (1/2-c)n-(3|F^+\cap C^+|/4+cn)-2\gamma^\prime n-127k^{3/2}\delta n\geq n/8.
\end{align}
Where the last inequality is derived from $|F^+\cap C^+|\leq|C|/2+7k^{3/2}\delta n$, $n\geq|A|+|B|+|C|$, $|A|\geq\frac{n}{7k^{3/2}}$ and $c\ll\gamma^\prime\ll\delta<\frac{1}{216k^3(k!)^2}$.

On the other hand, by Lemma \ref{lx8}-(i), $|Y^{i+}|\geq\frac{|Y^{(i-1)+}|}{2}-\epsilon n$ and $|Y^{(k-3)+}|\geq\frac{|Y|}{2^{k-3}}-2\epsilon n$.
According to (\ref{a4})-(\ref{a5}), in $F^-\cap C^-$, with the exception of at most $\gamma^\prime n$ vertices, all other vertices $z$ satisfy $d^+(z, A)\geq|A|-\gamma^\prime n$. Combining with the definition of $Y$, we get that $d^-(y, N^+(z, A))\geq|A|-2\gamma^\prime n$ for any $y\in Y$. Therefore, we obtain that
\begin{align}\label{a9}
&k\emph{\emph{-cyc}}(Y, Y^+,\ldots, Y^{(k-3)+}, F^-\cap C^-, A)\nonumber\\
&\geq|Y^{(k-3)+}|\cdot(\epsilon n)\cdot\prod_{i=0}^{k-4}(\epsilon n-i)\cdot(n/8-\gamma^\prime n)\cdot(|A|-2\gamma^\prime n)\nonumber\\
&\geq\frac{|C|-86k^{3/2}\delta n-2\gamma^\prime n}{2^{k-2}}\cdot(\epsilon n)^{k-3}\cdot(n/8-\gamma^\prime n)\cdot\frac{\eta n}{2}
\geq \alpha n^k,
\end{align}
where the first inequality is due to (\ref{a8}) and the second inequality is from (\ref{a6}) and (\ref{m3}) and $|A|\geq \frac{n}{7k^{3/2}}$.

Similarly, except for at most $\gamma^\prime n$ vertices, all other vertices $z^\prime$ in $F^-\cap C^-$ satisfy $d^+(z, B)\geq|B|-\gamma^\prime n$. By the definition of $Y$, we get $d^-(y^\prime, N^+(z, B))\geq|B|-2\gamma^\prime n$ for any $y^\prime\in Y$. Similar to (\ref{a9}), we can also get that
\begin{equation}\label{a10}
\begin{split}
k\emph{\emph{-cyc}}(Y, Y^+,\ldots, Y^{(k-3)+}, F^-\cap C^-, B)\geq \alpha n^k.
\end{split}
\end{equation}
Together with (\ref{a9}), (\ref{a10}) and $\mu\ll \alpha$, this implies $L_{\mathcal{P}_1}^\mu(H(S))$ contains a $2$-transferral, which contradicts our assumption.
\begin{Subcase}\label{subcase2}
Both $|F^+\cap C^-|\geq2\gamma^\prime n$ and $|F^-\cap C^+|\geq2\gamma^\prime n$.
\end{Subcase}
Due to (\ref{a4})-(\ref{a5}), except for at most $\gamma^\prime n$ vertices, all other vertices $x$ in $F^+\cap C^-$ satisfy that $d^-(x, B)\geq|B|-\gamma^\prime n$ and $d^+(x, A)\geq|A|-\gamma^\prime n$. Similarly, except for at most $\gamma^\prime n$ vertices, all other vertices $y$ in $F^-\cap C^+$ meet that $d^+(y, B)\geq|B|-\gamma^\prime n$ and $d^-(y, A)\geq|A|-\gamma^\prime n$. For the sake of convenience, for any such vertices $x$ and $y$, let $U_1=N^-(x, B)\cap N^+(y, B)$ and $U_2=N^+(x, A)\cap N^-(y, A)$, and then
\begin{equation}\label{e9}
\begin{split}
|U_1|\geq|B|-2\gamma^\prime n\ \emph{\emph{and}}\ |U_2|\geq|A|-2\gamma^\prime n.
\end{split}
\end{equation}
Also let $U_{1, 1}=\{u\in U_1:\ d^+(u, U_2)\geq|U_2|/2-2cn\}$ and $U_{1, 2}=\{u\in U_1:\ d^-(u, U_2)\geq|U_2|/2-2cn\}$.
It is not hard to get that either $|U_{1, 1}|\geq|U_1|/2$ or $|U_{1, 2}|\geq|U_1|/2$.
\begin{figure}[h]
\centering
\scriptsize

\bigskip
\begin{tabular}{ccc}
&\includegraphics[width=6.3cm]{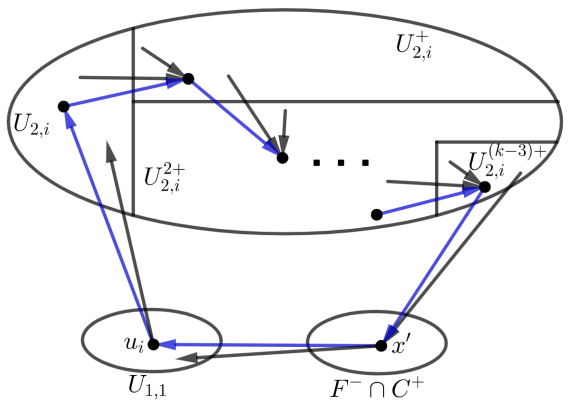}  \qquad\qquad\qquad\qquad&\includegraphics[width=6.3cm]{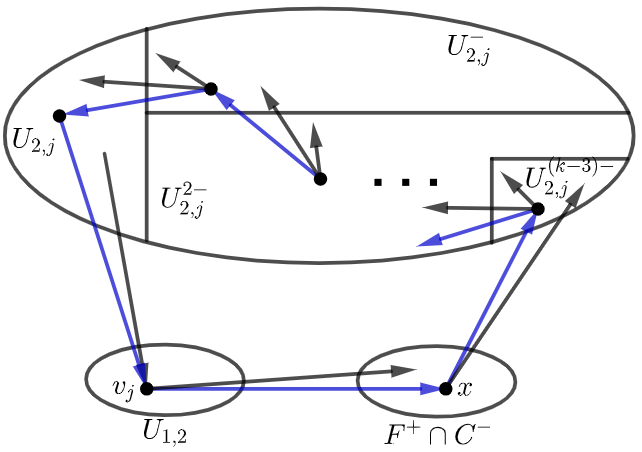}\\
& (a) The case when $|U_{1, 1}|\geq|U_1|/2$. \qquad\qquad\qquad\qquad& (b) The case when $|U_{1, 2}|\geq|U_1|/2$.
\end{tabular}

\caption{In the figure, the blue arcs form a $k$-cycle.}
\label{f-1}
\vspace{-0.5em}
\end{figure}

If $|U_{1, 1}|\geq|U_1|/2$, then for any $u_i\in U_{1, 1}$, let $U_{2, i}=N^+(u_i, U_2)$ and then
\begin{equation}\label{e10}
\begin{split}
 |U_{2, i}|\geq|U_2|/2-2cn\geq|A|/2-2\gamma^\prime n.
\end{split}
\end{equation}
Recall that
$U_{2, i}^{s+}=\{u\in U_{2, i}^{(s-1)+}: d^-(u, U_{2, i}^{(s-1)+})\geq \epsilon n\}$ for any $s\in[k-3]$.
Then by Lemma \ref{lx8}, we have that $|U_{2, i}^{s+}|\geq\frac{|U_{2, i}^{(s-1)+}|}{2}-\epsilon n$ and $|U_{2, i}^{(k-3)+}|\geq \frac{|U_{2, i}|}{2^{k-3}}-2\epsilon n$.
Thus by (\ref{e9})-(\ref{e10}), and $\alpha\ll\epsilon\ll\gamma^\prime\ll\eta$, we can get that (see Figure \ref{f-1} (a))
\begin{align}\label{e8}
k\emph{\emph{-cyc}}(F^-\cap C^+&, B, A^{k-2})\nonumber\\
&\geq(|F^-\cap C^+|-\gamma^\prime n)\cdot(|U_1|/ 2-\gamma^\prime n)\cdot(\epsilon n)\cdot\prod_{i=0}^{k-4}(\epsilon n-i)\cdot(|U_{1, i}^{(k-3)+}|-\gamma^\prime n)\nonumber\\
&\geq\gamma^\prime n\cdot(|B|/2-2\gamma^\prime n)\cdot(\epsilon n)^{k-3}\cdot\left(\frac{|A|-4\gamma^\prime n}{2^{k-2}}-\gamma^\prime n\right)\nonumber\\
&\geq\frac{\gamma^\prime \eta\epsilon^{k-3}}{2^{k-1}}\cdot n^k
\geq \alpha n^k,
\end{align}

If $|U_{1, 2}|\geq|U_1|/2$, similar to the above process, for any $v_j\in U_{1, 2}$, let $U_{2, j}=N^-(v_j, U_2)$. Then we deduce that
\begin{equation}\label{a13}
\begin{split}
|U_{2, j}|\geq|U_2|/2-2cn\geq|A|/2-2\gamma^\prime n.
\end{split}
\end{equation}
By Definition \ref{aaaaa}, we have that $U_{2, j}^{i-}=\{v: v\in U_{2, j}^{(i-1)-}\ \emph{\emph{and}}\ d^+(v, U_{2, j}^{(i-1)-})\geq \epsilon n\}$. Then by Lemma \ref{lx8}, we get that $|U_{2, j}^{i-}|\geq\frac{|U_{2, j}^{(i-1)+}|}{2}-\epsilon n$ and
$|U_{2, j}^{(k-3)+}|\geq\frac{|U_{2, j}|}{2^{k-3}}-2\epsilon n.$
Similarly, we can deduce that (see Figure \ref{f-1} (b))
\begin{equation*}
\begin{split}
k\emph{\emph{-cyc}}(F^+\cap C^-, A^{k-2}, B)&\geq(|F^+\cap C^-|-\gamma^\prime n)\cdot(\epsilon n)^{k-3}\cdot(|U_{2, j}^{(k-3)+}|-\gamma^\prime n)\cdot(|U_1|/2-\gamma^\prime n)\\
&\geq \gamma^\prime n\cdot(\epsilon n)^{k-3}\cdot\left(\frac{|A|-4\gamma^\prime n}{2^{k-2}}-\gamma^\prime n\right)\cdot(|B|/2-2\gamma^\prime n)\geq\alpha n^k.
\end{split}
\end{equation*}

Together with (\ref{e8}), this implies that there exists some partite set $R$ of $\mathcal{P}_1$ such that $k$-cyc$(R\cap(F^-\cap C^+), B, A^{k-2})\geq \mu n^k$ or $k$-cyc$(R\cap(F^+\cap C^-), A^{k-2}, B)\geq \mu n^k$. Recall that $k$-cyc$(A^{k-1}, B)\geq\alpha^\prime n^k$, so we get a $2$-transferral $\mathbf{v}_A-\mathbf{v}_R$ in $L_{\mathcal{P}_1}^\mu(H(S))$, which contradicts our assumption. Hence this claim holds.
\end{proof}
\begin{claim}\label{a2}
$\mathcal{P}_1$ is a $\delta$-extremal partition of $S$.
\end{claim}
\begin{proof}
It suffices to verify that $\mathcal{P}_1$ satisfies Definition \ref{definition3}.
Combining with Claim \ref{a1}, we know that there exists $A\in\mathcal{P}_1$ with $|A|\geq \frac{n}{7k^{3/2}}$ such that $k$-cyc$(A^{k-1}, S\setminus A)\leq\alpha n^k$.
For (not necessarily disjoint) subsets $U_1, U_2, \ldots, U_k, U_{k+1}$, we denote $L(U_1, U_2, \ldots, U_k, U_{k+1})$ to be the collection of $(k+1)$-sets $\{u_1, u_2, \ldots, u_k, u_{k+1}\}$ such that $u_i\in U_i$ for $i\in[k+1]$ and both $u_1u_2\cdots u_ku_1$ and $u_2\cdots u_{k}u_{k+1}u_2$ are $k$-cycles. We declare  that
\begin{equation}\label{k2}
\begin{split}
|L(A, V, \ldots, V, S\setminus A)|\leq \gamma_3 n^{k+1}.
\end{split}
\end{equation}
Otherwise, due to $|\mathcal{P}|\leq1/\eta$, there exist $U_2, \ldots, U_{k}\in\mathcal{P}$ and $U_{k+1}\in \mathcal{P}_1-A$ such that $|L(A, U_2, \ldots, U_{k}, U_{k+1})|\geq \eta^k\cdot \gamma_3n^{k+1}\geq k\mu n^{k+1}$. This implies that $k$-cyc$(A, U_2, \ldots, U_{k}), k$-cyc$(U_2, \ldots, U_{k}, U_{k+1})\geq\mu n^k$, which contradicts our conclusion that $L_{\mathcal{P}_1}^\mu(H(S))$ is $2$-transferral-free.

Set $S^\sigma=N_{\gamma_2, S}^\sigma(A)$ for $\sigma\in\{+, -\}$ and $R=S\setminus(S^+\cup S^-\cup A)$. By Lemma \ref{l14}, we have that
\begin{equation}\label{k3}
\begin{split}
|R|\leq2\gamma_3 n\ \mbox{and} \ |A|\leq(1/3+50k^3\delta)n.
\end{split}
\end{equation}
%
We first prove that Definition \ref{definition3}-(i) holds. Given $\sigma\in\{-, +\}$, suppose that there is a $k$-cycle $C=x\cdots yzx$ with $V(C)\setminus\{z\}\subseteq S^\sigma$ and $z\in S^{-\sigma}$.
Let $w\in A$ such that $V(C)\setminus \{y\}\cup\{w\}$ spans a $k$-cycle. Since $y\in S^\sigma$ and $z\in S^{-\sigma}$, we get that $d^{-\sigma}(y, A)\geq|A|-\gamma_2 n$ and $d^\sigma(z, A)\geq|A|-\gamma_2 n$. With (\ref{k2}), we see that
\begin{center}\label{k5}
$(|A|-2\gamma_2 n)\cdot k$-cyc$((S^\sigma)^{k-2}, S^\sigma, S^{-\sigma})\leq|L(A, (S^\sigma)^{k-2}, S^{-\sigma}, S^\sigma)|\leq\gamma_3 n^{k+1}$.
\end{center}
Together with $|A|\geq n/(7k^{3/2})$, we get that
\begin{center}
$k$-cyc$((S^\sigma)^{k-2}, S^\sigma, S^{-\sigma})\leq\gamma_2 n^k$.
\end{center}
Further,
$k$-cyc$((S^\sigma)^{k-2}, S^\sigma, A)\leq2\gamma_2 n\cdot\binom{|S^\sigma|}{k-1}\leq\gamma_2 n^k$,
and due to (\ref{k3}), we have that
$k$-cyc$((S^\sigma)^{k-2}, S^\sigma, R)\leq2\gamma_2 n^k$.
Hence,
\begin{equation*}
\begin{split}
k\text{-cyc}((S^\sigma)^{k-2}, S^\sigma, S\setminus S^\sigma)=k\text{-cyc}((S^\sigma)^{k-2}, S^\sigma, S^{-\sigma}\cup A\cup R)\leq4\gamma_2 n^k.
\end{split}
\end{equation*}
Applying Lemma \ref{l14} with $S^\sigma$ and $4\gamma_2$ taking on the roles of $A$ and $\alpha$, respectively, we get that
\begin{equation}\label{k6}
\begin{split}
|S^-|, |S^+|\leq(1/3+50k^3\delta)n.
\end{split}
\end{equation}
Set $A^\prime=A\cup R$, and then $\{A^\prime, S^-, S^+\}$ is a partition of $S$. It follows from (\ref{k3}) and (\ref{k6}) that $|A^\prime|, |S^-|, |S^+|\leq(1/3+50k^3\delta)n$. Thus we can deduce that
\begin{equation}\label{k7}
\begin{split}
(1/3-100k^3\delta)n\leq|A^\prime|, |S^-|, |S^+|\leq(1/3+50k^3\delta)n.
\end{split}
\end{equation}
Hence (i) of Definition \ref{definition3} holds. Now we show that Definition \ref{definition3}-(ii) also holds. Since $A^\prime=A\cup R$ and $|R|\leq 2\gamma_3 n$, for each $v\in S^\sigma$ with $\sigma\in\{-, +\}$, we get that $d^{-\sigma}(v, A^\prime)\geq|A|-\gamma_2 n\geq|A^\prime|-\eta n$.
This implies that
\begin{equation}\label{k8}
\begin{split}
e^\sigma(S^\sigma, A^\prime)\leq|S^\sigma|\cdot\gamma_1 n\leq\gamma_1 n^2.
\end{split}
\end{equation}
It follows by Lemma \ref{lx2}-(i) that $e^+(S^+, \overline{S^+})\geq e^-(S^+, \overline{S^+})-2cn^2\geq e^-(S^+, A^\prime)-2cn^2.$ By \ref{k7}, this implies
\begin{equation*}
\begin{split}
e^+(S^+, S^-)&=e^+(S^+, \overline{S^+})-e^+(S^+, A^\prime)-e^+(S^+, V_0)\\
&\geq(|A^\prime|-\gamma_1 n)|S^+|-2cn^2-\gamma_1 n^2-50k^3\delta n\cdot6k^{3/2}\delta n\geq |S^+|\cdot|S^-|-\delta n,
\end{split}
\end{equation*}
where the last inequality is derived from the fact that $\delta<\frac{1}{216k^{3}(k!)^2}$.
Thus $e^+(S^-, S^+)\leq\delta n$. Together with (\ref{k7})-(\ref{k8}), we get that $\mathcal{P}$ is a $\delta$-extremal partition of $S$.
\end{proof}
\subsection{Completion of Theorem \ref{t1}}
By Claim \ref{a1}, there exists a partite set $A\in\mathcal{P}_1$ with $|A|\geq \frac{n}{7k^{3/2}}$ and  $k$-cyc$(A^{k-1}, S\setminus A)\leq\alpha n^k$. Set $S^\sigma=N_{\gamma_2, S}^\sigma(A)$ for $\sigma\in\{+, -\}$. Let $A_2^\prime$ (resp., $A_0^\prime$) be the partite set of $\mathcal{P}_1$ such that $|A_2^\prime\cap S^+|$ (resp., $|A_0^\prime\cap S^-|$) is maximum and define $A_2=A_2^\prime\cap S^+$ and $A_0=A_0^\prime\cap S^-$. Since $(1/3-100k^3\delta)n\leq|S^+|, |S^-|\leq(1/3+50k^3\delta)n$ and $d\leq6k^{3/2}$, we have that $|A_2|, |A_0|\geq\frac{n}{20k^{3/2}}$. Further, by Claim \ref{a2}, $A_2\subseteq S^+$ and $A_0\subseteq S^-$, we get that,
\begin{equation*}
\begin{split}
&d^-(x, A)\geq|A|-\gamma_2n\ \mbox{for any}\ x\in A_2,\ \mbox{and}\ e^+(A_2, A)\leq\delta n^2,\
\mbox{and} \\
&d^+(y, A)\geq|A|-\gamma_2n\ \mbox{for any}\ y\in A_0, e^-(A, A_0)\leq\delta n^2 \ \mbox{and}\ e^+(A_0, A_2)\leq\delta n^2.
\end{split}
\end{equation*}
Then by Lemma \ref{l7} with $\delta$, $A_2$ and $A$ (resp., $\delta$, $A_0$ and $A_2$) taking on the roles of $\gamma_2$, $A$ and $B$, respectively, we obtain that
\begin{equation}\label{q2}
\begin{split}
&|N_{\delta_1}^-(A_2)\cap A|\geq|A|-\frac{\delta_1 n}{2} \ \emph{\emph{and}}\  |N_{\delta_1}^+(A)\cap A_2|\geq|A_2|-\frac{\delta_1 n}{2},\ \mbox{and} \\
&|N_{\delta_1}^-(A_0)\cap A_2|\geq|A_2|-\frac{\delta_1 n}{2}\  \emph{\emph{and}}\  |N_{\delta_1}^+(A_2)\cap A_0|\geq|A_0|-\frac{\delta_1 n}{2}.
\end{split}
\end{equation}
For consistency, we define $A_1:=A$. Combining with the definitions of $A_2$ and $A_0$, we get that for every $i\in\{0, 1, 2\}$,
\begin{equation}\label{q4}
\begin{split}
\left|\left\{z\in A_i: d^-(z, A_{i-1})\geq|A_{i-1}|-\frac{\delta_1 n}{2}\ \emph{\emph{and}}\ d^+(z, A_{i+1})\geq|A_{i+1}|-\frac{\delta_1 n}{2}\right\}\right|\geq|A_i|-\delta_1 n,
\end{split}
\end{equation}
here and subsequently, the subscripts of $A_{i-1}$ and $A_{i+1}$ take the remainder modulo $3$.
Without loss of generality, assume that $|A_2|\geq|A_0|$. In the following, we consider three cases in terms of the value of the remainder of $k$ modulo $3$. In each case, we will get that $L_{\mathcal{P}_1}^\mu(H(S))$ contains a $2$-transferral.

\begin{figure}[h]
\centering
\scriptsize
%
%

\begin{tabular}{ccc}\label{zzz}
\includegraphics[width=4.3cm]{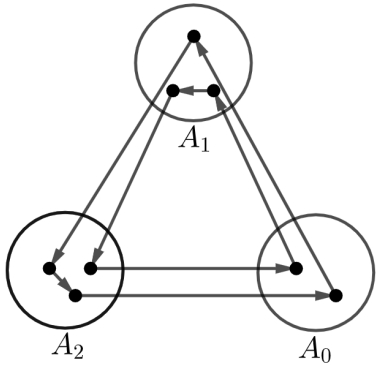}&\includegraphics[width=4.3cm]{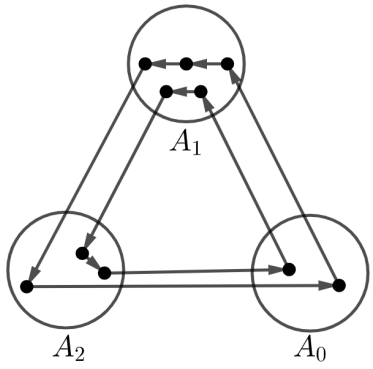}&\includegraphics[width=4.3cm]{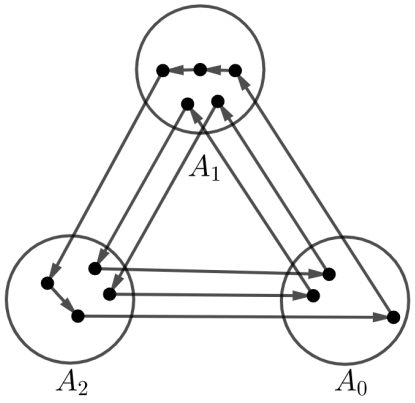}\\
(a) $k=4$ for Case $1$. &(b) $k=5$ for Case $2$. &(c) $k=6$ for Case $3$.
\end{tabular}

%
\caption{
%
Examples for Cases 1-3.}
\label{Examples}
\vspace{-0.5em}
\end{figure}

\smallskip

\noindent \textbf{Case 1.}
$k\equiv1$ (mod $3$).

\medskip
In this case, let $\mathcal{C}_1^1$ be the family of $k$-cycles $u_1\cdots u_{k}u_1$ such that
for every $j\in[k]$, 
\begin{align}\label{aa1}
u_j\in A_i\ \mbox{when}\ j\equiv i\ (\emph{\emph{mod}}\ 3).
\end{align}
Similarly, let $\mathcal{C}_1^2$ be the family of $k$-cycles $u_1\cdots u_{k}u_1$, where for $j\in[k]$, $u_j$ satisfies that
\begin{align}\label{aa2}
u_j\in A_{i+1}\ \emph{\emph{when}}\ j\equiv i\ (\emph{\emph{mod}}\ 3).
\end{align}
We first calculate the size of $\mathcal{C}_1^1$. By Lemma \ref{a3}, there exists $A_1^\prime\subseteq A_1$ with $|A_1^\prime|\geq|A_1|/2$ such that $d^-(x, A_1)\geq|A_1|/4-cn$ for any $x\in A_1^\prime$. Pick any $u_1\in A_1^\prime$. To construct $k$-cycles through $u_1$, we pick paths $u_1u_2\cdots u_{k-1}$ that starting at $u_1$ and $u_j$ meeting (\ref{aa1}) for $j\in[k-1]$. By (\ref{q4}), the number of such $(k-1)$-paths is at least
\begin{align}\label{u1}
&(|A_2|-\delta_1 n)(|A_0|-\delta_1 n)\cdot\prod\limits_{i=1}^{(k-4)/3}((|A_1|-\delta_1 n-i)(|A_2|-\delta_1 n-i)(|A_0|-\delta_1 n-i))\nonumber\\
&\geq\frac{n^{k-2}}{3^kk^{k-1}},
\end{align}
where the inequality is due to $|A_1|=|A|\geq \frac{n}{3}-100k^3\delta n$, $|A_2|, |A_0|\geq \frac{n}{20k^{3/2}}$ and $\delta_1<\frac{1}{20k^2}$.
Given such a $(k-1)$-path $u_1\cdots u_{k-1}$, we may make it a $k$-cycle by selecting the vertex $u_k$ in $N^{-}(u_1, A_1)\cap N^+(u_{k-1}, A_1)$. Clearly there exist at least $|A_1|/4-cn-\delta_1 n-(k-1)/3$ such vertices $u_k$. Hence by (\ref{u1}) we can get that the number of $k$-cycles satisfying (\ref{aa1}) is at least
\begin{equation}\label{u2}
\begin{split}
\frac{1}{k}\cdot(|A_1|-\delta_1 n)\cdot\frac{n^{k-2}}{3^kk^{k-1}}\cdot\left(\frac{|A_1|}{4}-cn-\delta_1 n-\frac{k-1}{3}\right)\geq\mu n^k.
\end{split}
\end{equation}

Next we calculate the size of $\mathcal{C}_1^2$. By Lemma \ref{a3}, there exists vertex subset $A_2^\prime\subseteq A_2$ with $|A_2^\prime|\geq|A_2|/2$ so that for any $y\in A_2^\prime$, we obtain that $d^-(y, A_2)\geq|A_2|/4-cn$. With $A_2$, $A_3$, $A_1$ and $A_2^\prime$ taking on the roles of $A_1$, $A_2$, $A_3$ and $A_1^\prime$ respectively in the process of calculating $\mathcal{C}_1^1$ described above, we can get that $|\mathcal{C}_1^2|\geq\mu n^k$. Together with $|\mathcal{C}_1^1|\geq\mu n^k$, this suggests that $L_{\mathcal{P}_1}^\mu(H(S))$ contains a $2$-transferral, which contradicts our assumption. (We give an example when $k=4$ for Case $1$ in Figure \ref{Examples} (a).)

\medskip

\noindent \textbf{Case 2.}
$k\equiv2$ (mod $3$).

\medskip
Let $\mathcal{C}_2^1$ be the family of $k$-cycles $u_1\cdots u_{k}u_1$ such that
\begin{align}\label{aa3}
u_1\in A_1\ \mbox{and for any}\ j\in\{2, 3, \ldots, k\}, u_j\in A_{i-1}\ \emph{\emph{when}}\ j\equiv i\ (\emph{\emph{mod}}\ 3).
\end{align}
Likewise, let $\mathcal{C}_2^2$ be the family of $k$-cycles $u_1\cdots u_{k}u_1$ such that
\begin{align}\label{aa4}
\mbox{for}\ j\in[2], u_j\in A_j,\ \mbox{and for any}\ j\in\{3, 4, \ldots, k\}, u_j\in A_{i-1}\ \emph{\emph{when}}\ j\equiv i\ (\emph{\emph{mod}}\ 3).
\end{align}
We first prove that $|\mathcal{C}_2^1|\geq\mu n^k$. Again by Lemma \ref{a3}, there exists $A_1^\prime\subseteq A_1$ with $|A_1^\prime|\geq|A_1|/2$ such that $d^-(x, A_1)\geq|A_1|/4-cn$ for any $x\in A_1^\prime$. Let $A^{\prime\prime}_1=\{x\in A_1^\prime: d^-(x, A_1^\prime)\geq\epsilon n\}$, and then by Lemma \ref{lx8}, we can get that $|A^{\prime\prime}_1|\geq|A_1^\prime|/2-\epsilon n$. Similar to the calculation process for $\mathcal{C}_1^1$, by (\ref{q4}), the size of $\mathcal{C}_2^1$ is at least
\begin{equation*}
\begin{split}
&\frac{1}{k}\cdot\left(\frac{|A_1|}{4}-cn-\delta_1n-2-\frac{k-5}{3}\right)\cdot\epsilon n\cdot(|A^{\prime\prime}_1|-\delta_1 n)\cdot(|A_2|-\delta_1 n)(|A_0|-\delta_1 n)\cdot\\
&\prod\limits_{i=1}^{(k-5)/3}(|A_1|-2\delta_1n-1-i)(|A_2|-\delta_1 n-i)(|A_0|-\delta_1 n-i),
\end{split}
\end{equation*}
which is larger than $\mu n^k$ since $|A_1|=|A|\geq \frac{n}{3}-100k^3\delta n$, $|A_2|, |A_0|\geq \frac{n}{20k^{3/2}}$, $|A_1^{\prime\prime}|\geq|A_1^\prime|/2-\epsilon n\geq|A_1|/4-\epsilon n$, and $c\ll\mu\ll\delta_1<\frac{1}{20k^2}$.

We next estimate the size of $\mathcal{C}_2^2$. For any $u_1\in A_1^\prime$, let $A_2^\prime=N^+(u_1, A_2)$, and then $|A_2^\prime|\geq|A_2|-\delta_1 n/2$ by (\ref{q2}). Define $A_2^{\prime\prime}=\{y\in A_2^\prime: d^-(y, A_2^\prime)\geq\epsilon n\}$, and then $|A_2^{\prime\prime}|\geq|A_2^{\prime}|/2-\epsilon n$ via Lemma \ref{lx8}. In addition, there exist at least $|A_2^{\prime\prime}|-\delta_1 n/2$ vertices in $|A_2^{\prime\prime}|$ such that each of these vertices $u_3$ satisfying $d^+(u_3, A_0)\geq|A_0|-\delta_1 n/2$, and there exist at least $|A_0|-\delta_1 n$ vertices in $A_0$ so that each of these vertices $u_4$ satisfying $d^+(u_4, A_1)\geq|A_1|-\delta_1 n/2$. Greedily, we can get the path $u_1\cdots u_{k-1}$ and the number of vertices $u_k$ in $N^-(u_1, A_1)$ with $u_{k-1}u_k\in A$ is at least $|N^-(u_1, A_1)|-\delta_1 n/2-(k-5)/3$. Hence the size of $\mathcal{C}_2^2$ is at least
\begin{equation*}
\begin{split}
&\frac{1}{k}\cdot(|A_1|/4-cn-\delta_1 n-(k-5)/3)\cdot(|A_1^\prime|-\delta_1 n)\cdot\epsilon n\cdot(|A_2^{\prime\prime}|-\delta_1 n)\cdot(|A_0|-\delta_1 n)\cdot\\
&\prod\limits_{i=1}^{(k-5)/3}((|A_1|-\delta_1-i)(|A_2|-\delta_1 n-1-i)(|A_0|-\delta_1 n-i)),
\end{split}
\end{equation*}
which is more than $\mu n^k$ since $|A_1|\geq \frac{n}{3}-100k^3\delta n$, $|A_2|, |A_0|\geq \frac{n}{20k^{3/2}}$, $|A_1^{\prime\prime}|\geq|A_1|/4-\epsilon n$, $|A_2^{\prime\prime}|\geq|A_2^\prime|/2-\epsilon n\geq|A_2|/2-\delta_1n/4-\epsilon n$, and $c\ll\mu\ll\delta_1<\frac{1}{20k^2}$. Combining with $|\mathcal{C}_2^1|\geq\mu n^k$, we can get a $2$-transferral, which contradicts that $L_{\mathcal{P}_1}^\mu(H(S))$ is $2$-transferral-free (we give an example when $k=5$ for Case $2$ in Figure \ref{Examples} (b)).

\medskip

\noindent \textbf{Case 3. }
$k\equiv0$ (mod $3$).

\medskip
Let $\mathcal{C}_3^1$ be the family of $k$-cycles $u_1\cdots u_{k}u_1$ such that for any $j\in[k]$
\begin{align}\label{aa5}
u_j\in A_i\ \mbox{when}\ j\equiv i\ (\emph{\emph{mod}}\ 3).
\end{align}
Similarly, $\mathcal{C}_3^2$ is the family of $k$-cycles $u_1\cdots u_{k}u_1$, where $u_1, u_2\in A_1, u_3, u_4\in A_2$ and
\begin{align}\label{aa6}
\mbox{for any}\ j\in\{5, \ldots, k\}, u_j\in A_{i+1}\ \mbox{when}\ j\equiv i\ (\emph{\emph{mod}}\ 3).
\end{align}
It is not hard to get from (\ref{q4}) that
\begin{equation*}
\begin{split}
|\mathcal{C}_3^1|\geq&(|A_1|-\delta_1n-(k-3)/3)\cdot(|A_2|-\delta_1n)\cdot(|A_0|-\delta_1 n)\cdot\\
&\prod\limits_{i=1}^{(k-3)/3}(|A_1|-\delta_1n-i)(|A_2|-\delta_1 n-1-i)(|A_0|-\delta_1 n-i)\geq\mu n^k.
\end{split}
\end{equation*}
Also, similar to the estimations of $|\mathcal{C}_2^1|$ and $|\mathcal{C}_2^2|$, there exist $A_1^\prime, A_1^{\prime\prime}\subseteq A_1$ with $|A_1^\prime|\geq|A_1|/2$ and $|A_1^{\prime\prime}|\geq|A_1^\prime|/2-\epsilon n$ such that $d^-(x, A_1)\geq|A_1|/4-cn$ and $d^-(y, A_1^\prime)\geq\epsilon n$ for any $x\in A_1^\prime$ and $y\in A^{\prime\prime}_1$, respectively. At the same time, For any $u_2\in A_1^{\prime\prime}$, there exist $A_2^\prime, A_2^{\prime\prime}\subseteq A_2$ with $|A_2^\prime|\geq|A_2|-\delta_1 n/2$ and $|A_2^{\prime\prime}|\geq|A_2^\prime|/2-\epsilon n$ such that $A_2^\prime=N^+(u_2, A_2)$ and $A_2^{\prime\prime}=\{y\in A_2^\prime: d^-(y, A_2^\prime)\geq\epsilon n\}$. Clearly, there exist at least $|A_1^{\prime\prime}|-\delta_1 n/2$ vertices $u_4$ in $|A_1^{\prime\prime}|$ such that $d^+(u_4, A_0)\geq|A_0|-\delta_1 n/2$. Further, similar to $|\mathcal{C}_2^2|$, we can proceed by greedily finding vertices $u_5, \ldots, u_k$ until we obtain a $k$-cycle $u_1\cdots u_ku_1$. Hence, we get
\begin{equation*}
\begin{split}
|\mathcal{C}_3^2|\geq& \epsilon n\cdot(|A_1^\prime|/2-\epsilon n-\delta_1 n/2)\cdot\epsilon n\cdot(|A_2^{\prime}|/2-\epsilon n-\delta_1 n/2)\cdot(|A_0|-\delta_1 n)\cdot\\
&\prod\limits_{i=1}^{(k-6)/3}(|A_1|-\delta_1n-1-i)(|A_2|-\delta_1 n-1-i)(|A_0|-\delta_1n-i)\cdot\\
&((|A_1|/4-cn-\delta_1 n-(k-6)/3))\geq\mu n^k.
\end{split}
\end{equation*}
Together with $|\mathcal{C}_3^2|\geq\mu n^k$, this contradicts what we already know about that $L_{\mathcal{P}_1}^\mu(H(S))$ is $2$-transferral-free (we give an example when $k=6$ under this case in Figure \ref{Examples} (c)). Thus we complete the proof of Theorem \ref{t1}. \hfill $\Box$
\section*{Acknowledgments}
We would especially like to thank Jie Han for helpful discussion and valuable comments. We also thank the referee for the careful readings and suggestions that improve the presentation.

\end{document}